\documentclass{amsproc}
\usepackage{amsmath,amssymb,latexsym,amsfonts,cite,hyperref, enumerate, graphicx, fullpage}
\usepackage[all]{xy}

\hypersetup{unicode=false, pdftoolbar=true, pdfmenubar=true, pdffitwindow=false, pdfstartview={FitH}, pdfnewwindow=true, colorlinks=false, linkcolor=blue, citecolor=green}

\newtheorem{theorem}{Theorem}[section]
\newtheorem{proposition}[theorem]{Proposition}
\newtheorem{conjecture}[theorem]{Conjecture}
\newtheorem{corollary}[theorem]{Corollary}
\newtheorem{lemma}[theorem]{Lemma}
\newtheorem{definition/proposition}[theorem]{Definition/Proposition}

\newtheorem{definitions}[theorem]{Definitions}

\theoremstyle{definition}
\numberwithin{equation}{section}

\begin{document}
\title[Differential-Dunkl Operators and Nonstandard Solutions to the Classical Yang-Baxter Equation]{Differential-Dunkl Operators and Nonstandard Solutions to the Classical Yang-Baxter Equation}
\author{Garrett Johnson}
\address{Department of Mathematics, The Catholic University of America, Washington, D.C. 20064}
\email{johnsongw@cua.edu}
\subjclass[2010]{16T20; 16S40}
\keywords{Cherednik algebras, R-matrices, Yang-Baxter equation}

\begin{abstract}

For every pair of positive coprime integers, $m$ and $n$, with $m<n$, there is an associated generalized Cremmer-Gervais $r$-matrix $r_{m,n}\in\mathfrak{sl}_n\wedge\mathfrak{sl}_n$ which provides a nonstandard quasitriangular solution to the classical Yang-Baxter equation. We give an interpretation of $r_{n-2,n}$ (for $n$ odd)  in terms of differential-Dunkl operators related to the polynomial representation of dihedral-type rational Cherednik algebras. Finally, we use this interpretation to partially answer a conjecture of Gerstenhaber and Giaquinto concerning boundary solutions to the classical Yang-Baxter equation.
\end{abstract}

\maketitle

\section{Introduction} 

In the early 1980's, Belavin and Drinfeld \cite{BD} classified the quasitriangular solutions to the classical Yang-Baxter equation. The $r$-matrices $r_{m,n}\in\mathfrak{sl}_n\wedge\mathfrak{sl}_n$ of Cremmer and Gervais (see e.g. \cite{CG,EH,GG}), which are indexed by pairs of positive coprime integers $m$ and $n$ with $m<n$,  provide an interesting family of such solutions. 

In this paper, we give an interpretation of the Cremmer-Gervais $r$-matrices $r_{n-2,n}$ (for $n$ odd) in terms of differential-Dunkl operators related to the polynomial representation of certain dihedral-type rational Cherednik algebras. The rational Cherednik algebras are a family of algebras defined by Etingof and Ginzburg \cite{EG} in the context of symplectic reflection algebras. They are doubly degenerate versions of double affine Hecke algebras, which were first introduced by Cherednik \cite{Ch} as a key tool to solving the Macdonald constant term conjectures, and have since been of importance in noncommutative algebras and quantum groups. For a more detailed exposition of these algebras, we refer the reader to \cite{EM}.

The organization of this paper is as follows. Section \ref{CYB section} covers the preliminaries on classical $r$-matrices, the Belavin-Drinfeld classification theorem, and the conjecture of Gerstenhaber and Giaquinto \cite[Conj. 5.7]{GG} concerning Cremmer-Gervais $r$-matrices and boundary solutions to the classical Yang-Baxter equation.   In Section \ref{differential-Dunkl operators} we introduce an algebra of differential-Dunkl operators related to the polynomial representation of dihedral-type rational Cherednik algebras and give an interpretation of certain Cremmer-Gervais $r$-matrices in terms of these operators. In Section \ref{GG Conj Section}, we prove the Gerstenhaber-Giaquinto conjecture for the case when $m=n-2$ (and $n$ is necessarily odd). Section \ref{33} is dedicated to a proof of Thm. \ref{the action}, which gives an explicit formula for the Cremmer-Gervais $r$-matrices. \\

\noindent{\bf Acknowledgments.} This paper is based on some work which began as part of my Ph.D. thesis \cite{JThesis} at the University of California at Santa Barbara. I would like to thank my thesis advisor, Milen Yakimov, for his guidance. I would also like to thank Anthony Giaquinto for communicating his proof of part 1 of Thm. \ref{GGConj} to me.

\section{Classical $r$-matrices}\label{CYB section}

\subsection{Preliminaries and definitions} 
Throughout this paper, $m$ and $n$ denote fixed positive integers with $m<n$. Later, we make the further assumption that $m$ and $n$ are coprime. Let $k$ be an algebraically closed field of characteristic zero.  We mention here, however, that most of the statements and results of this paper hold over a field with $2n\neq 0$ and containing an $m$-th root of unity. All tensor products are taken over $k$.  For a $k$-module $V$, and vectors $a,b\in V$, we use the notation $a\wedge b$ as shorthand for $\frac{1}{2}\left(a\otimes b-b\otimes a\right)$ and let $V\wedge V$ be the vector subspace of $V\otimes V$ spanned by the set $\{a\wedge b: a,b\in V\}$.

Let $\mathfrak{g}$ be a simple Lie algebra over $k$ and let $U(\mathfrak{g})$ denote its universal enveloping algebra. For $r=\sum_ia_i\otimes b_i\in\mathfrak{g}\otimes\mathfrak{g}$, define the following elements of $U(\mathfrak{g})\otimes U(\mathfrak{g})\otimes U(\mathfrak{g})$: $r_{12}=\sum_i a_i\otimes b_i\otimes 1$, $r_{23}=\sum_i1\otimes a_i\otimes b_i$, and $r_{13}=\sum_ia_i\otimes 1\otimes b_i$. We call $r\in\mathfrak{g}\wedge\mathfrak{g}$ an \emph{$r$-matrix} if 
\begin{equation}
\langle\langle r,r\rangle\rangle := [r_{12},r_{13}]+[r_{12},r_{23}]+[r_{13},r_{23}]
\end{equation}
is $\mathfrak{g}$-invariant. An $r$-matrix is called \emph{triangular} if $\langle\langle r,r\rangle\rangle =0$ and \emph{quasitriangular} if $\langle\langle r,r\rangle\rangle\neq0$. Although we don't discuss Poisson-Lie groups or Lie bialgebras in this paper, we mention here that $r$-matrices arise in the context of determining Poisson structures on Lie groups which are compatible with the group multiplication. For further reference, see e.g. \cite[Chapter 1]{CP}.

In this paper we turn our attention to the case when  $\mathfrak{g}=\mathfrak{sl}_n$, the Lie algebra of traceless $n\times n$ matrices. Let $e_{ij}$ denote the elementary $n\times n$ matrix having $1$ in the $(i,j)$-entry and zeros elsewhere. We identify $e_{ij}$ with the $k$-linear map on $k^n$ defined by $e_{ij}e_\ell=\delta_{j\ell}e_i$, where $e_1,...,e_n$ denote the standard basis vectors of $k^n$. Under this identification, we view an element $r\in{\mathfrak sl}_n\wedge\mathfrak{sl}_n$ as a linear operator on $k^n\otimes k^n$ that satisfies $PrP=-r$, where $P$ is the linear map that interchanges tensor components: $u\otimes v\mapsto v\otimes u$. It turns out that the $\mathfrak{sl}_n$-invariants in $U(\mathfrak{sl}_n)\otimes U(\mathfrak{sl}_n)\otimes U(\mathfrak{sl}_n)$ correspond to scalar multiples of the linear map given by $u\otimes v\otimes w\mapsto w\otimes u\otimes v - v\otimes w\otimes u$ for all $u,v,w\in k^n$. This prompts the following definitions.

\begin{definitions} Let $V$ be a $k$-module, and let $\tau:V\otimes V\to V\otimes V$ be a linear map satisfying $P\tau P=-\tau$.
\begin{enumerate}[(1)]
\item For an arbitrary (but fixed) scalar $\lambda\in k$, define the classical Yang-Baxter map $CYB_\lambda :\text{End}_k(V^{\otimes 2})\to\text{End}_k(V^{\otimes 3})$ as follows,
\begin{equation}\label{CYB equation}
CYB_\lambda(\tau):=[\tau_{12},\tau_{13}]+[\tau_{12},\tau_{23}]+[\tau_{13},\tau_{23}]-\lambda Z,
\end{equation} 
where $Z$ is the linear map defined by $u\otimes v\otimes w\mapsto w\otimes u\otimes v - v\otimes w\otimes u$ for all $u,v,w\in V$.
\item We call $\tau$ an $r$-matrix if there exists $\lambda\in k$ so that $CYB_\lambda(\tau)=0$.  If $CYB_0(\tau)=0$ then $\tau$ is called a triangular $r$-matrix.  An $r$-matrix with $CYB_0(\tau)\neq 0$ is called quasitriangular.
\end{enumerate}
\end{definitions}

\subsection{Quasitriangular $r$-matrices, Belavin-Drinfel'd Classification}\label{Quasitriangular $r$-matrices section} 
A classification of quasitriangular $r$-matrices for the case when $\mathfrak{g}$ is a complex simple Lie algebra was completed in the early 1980's by Belavin and Drinfel'd \cite{BD}.

In this section, we assume $\mathfrak{g}$ is complex and simple. Let $\mathfrak{h}\subseteq\mathfrak{g}$ denote a  Cartan subalgebra, and let $\left<\,\,,\,\,\right>$ be a nondegenerate symmetric ad-invariant bilinear form on $\mathfrak{g}$. The nondegeneracy of the bilinear form induces an inner product on the dual vector space ${\mathfrak g}^*$, also denoted with $\left<\,\,,\,\,\right>$. For a root system $\Phi=\Phi_+\sqcup\Phi_-$ of $(\mathfrak{g},\mathfrak{h})$, let $\Pi=\{\alpha_1,...,\alpha_\ell\}\subseteq\Phi_+$ denote a basis of positive roots of $\mathfrak{g}$, and let $\mathfrak{g}\cong\mathfrak{n}_-\oplus\mathfrak{h}\oplus\mathfrak{n}_+$ denote the corresponding triangular decomposition. For each positive root $\alpha\in\Phi_+$, select root vectors $e_\alpha\in{\mathfrak g}_\alpha, e_{-\alpha}\in{\mathfrak g}_{-\alpha}$  so that $\left< e_\alpha ,e_{-\alpha}\right> =1$. For each root $\beta\in\Phi$, let $h_{\beta}\in\mathfrak{h}$ denote the unique vector satisfying $\left< h_\beta ,H\right> = \alpha (H)$ for all $H\in\mathfrak{h}$. A \emph{BD-triple} $({\mathcal S}_0,{\mathcal S}_1,\zeta)$ is a bijection $\zeta : {\mathcal S}_0\to {\mathcal S}_1$ between subsets ${\mathcal S}_0\subseteq\Pi$ and ${\mathcal S}_1\subseteq \Pi$ and satisfies 
\begin{enumerate}[(1)]
\item the orthogonality condition: $\left<\zeta (\alpha_i ),\zeta (\alpha_j)\right>=\left<\alpha_i,\alpha_j\right>$ for every $\alpha_i,\alpha_j\in {\mathcal S}_0$, and
\item the nilpotency condition: for every $\alpha\in {\mathcal S}_0$, there exists $N\in\mathbb{N}$ so that $\zeta^N(\alpha)\in {\mathcal S}_1\backslash {\mathcal S}_0$.
\end{enumerate}
We view BD-triples graphically via two copies of the Dynkin diagram of $\mathfrak{g}$ with arrows representing the bijection $\zeta$ drawn between the nodes (see, e.g., Figures \ref{CG BD-triple T_1n} and \ref{CG BD-triple T_n-2,n}).  

\begin{figure}[h]\centering\caption{The Cremmer-Gervais BD-triple ${\mathcal T}_{1,n}$}\begin{picture}(300,50)\label{CG BD-triple T_1n}\put(18,35){\xymatrix @R=10pt @C=10pt{&*={\bullet}\ar[ddrr]\ar@{-}[rrrrrr]&&*={\bullet}\ar[ddrr]&&*={\bullet}\ar[ddrr]&&*={\bullet}\ar@{-}[r]&\cdots&\!\!\!\!\!\!\!\cdots&*={\bullet}\ar@{-}[l]\ar@{-}[rrrrrr]\ar[ddrr]&&*={\bullet}\ar[ddrr]&&*={\bullet}\ar[ddrr]&&*={\bullet}\\&&&&&&&&\cdots\ar@{-}[ul]&*={\!\!\!\!\!\!\!\cdots}\ar[dr]\\&*={\bullet}\ar@{-}[rrrrrr]&&*={\bullet}&&*={\bullet}&&*={\bullet}\ar@{-}[r]&\cdots&\!\!\!\!\!\!\!\cdots&*={\bullet}\ar@{-}[l]\ar@{-}[rrrrrr]&&*={\bullet}&&*={\bullet}&&*={\bullet}\\}}\end{picture}\end{figure}
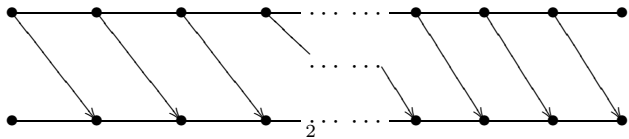

For a BD-triple ${\mathcal T}=({\mathcal S}_0,{\mathcal S}_1,\zeta)$, we put
\begin{equation}
\beta_{\mathcal T}:=\Big\{\beta\in\mathfrak{h}\wedge\mathfrak{h}:(1\otimes (\zeta(\alpha)-\alpha))\beta=\frac{1}{2}(h_{\zeta(\alpha)}+h_{\alpha})\text{ for all }\alpha\in {\mathcal S}_0\Big\}.
\end{equation} 
It turns out that $\beta_{\mathcal T}$ is an affine algebraic subvariety of $\mathfrak{h}\wedge\mathfrak{h}$ of dimension $d(d-1)/2$ where $d=\#(\Pi-{\mathcal S}_0)$ (see, e.g., \cite[Section 3.1]{CP}). For $i\in\{0,1\}$, put $\widehat{{\mathcal S}}_i:=\mathbb{Z}{\mathcal S}_i\cap\Phi_+$, and let $\widehat{\zeta}$ denote the $\mathbb{Z}$-linear extension of $\zeta$ to $\widehat{{\mathcal S}}_0$. The BD-triple induces a partial ordering $\preceq$ on the positive roots as follows: for $\rho,\mu\in\Phi_+$, $\rho\prec\mu$ if and only if there exists $N\in\mathbb{N}$ so that $\widehat{\zeta}^N(\rho)=\mu$.

It follows from the definition of quasitriangularity that if $r$ is quasitriangular, then $\lambda r$ is also quasitriangular for every $\lambda\in\mathbb{C}^\times$. Since the defining equations for an $r$-matrix are quadratic, we can view the set of all quasitriangular $r$-matrices as a projective subvariety ${\mathcal M}$ of $\mathbb{P}(\mathfrak{g}\wedge\mathfrak{g})$. We also observe that the group $\text{Aut}_\mathfrak{g}(\mathfrak{g})\times\mathbb{C}^\times$ acts on ${\mathcal M}$ by $(g,\lambda).r:=\lambda (g\otimes g)r$. We call $r$ and $r^\prime$ \emph{equivalent} $r$-matrices if there exists a nonzero scalar $\lambda\in\mathbb{C}^\times$ and an \emph{inner} automorphism $g\in\text{Int}_{\mathfrak g}(\mathfrak{g})$ so that $(g,\lambda).r=r^\prime$. Now we are ready to state the Belavin-Drinfel'd classification theorem.

\begin{theorem} (Belavin-Drinfel'd Classification Theorem) 
For a $BD$-triple ${\mathcal T}$ and $\beta\in\beta_{\mathcal T}$, 
\begin{equation}\label{BD Classification}
r_{{\mathcal T},\beta}=\alpha_{\mathcal T}+\beta+\gamma\in{\mathfrak g}\wedge{\mathfrak g}
\end{equation}
is a quasitriangular $r$-matrix on $\mathfrak{g}$, where
\begin{equation}
\alpha_{\mathcal T}=2\sum_{\rho\prec\mu}e_\rho\wedge e_{-\mu}
\end{equation}
and 
\begin{equation}
\gamma = \sum_{\mu\in\Phi_+} e_\mu\wedge e_{-\mu}.  
\end{equation}
Conversely, any quasitriangular $r$-matrix $r\in\mathfrak{g}\wedge\mathfrak{g}$ is equivalent to a unique $r$-matrix of the above form.
\end{theorem}

\subsection{Generalized Cremmer-Gervais $r$-matrices}\label{Generalized CG $r$-matrices section}

The generalized Cremmer-Gervais $r$-matrices provide some interesting examples of quasitriangular $r$-matrices for the case when $\mathfrak{g}=\mathfrak{sl}_n$.

The linear span of $\{e_{ii}-e_{i+1,i+1}: 1\leq i \leq n-1\}$ is a Cartan subalgebra, which we denote as $\mathfrak{h}$. The bilinear form we use is the trace form $\left<X,Y\right>:=Tr(XY)$.  The positive root vectors are $\{ e_{ij}: i<j\}$ and the negative root vectors are $\{e_{ij}: i>j\}$. Let $\alpha_1,...,\alpha_{n-1}$ denote the simple roots, where $\alpha_i$ is the linear functional on $\mathfrak{h}$ defined by the rule $\alpha_i(h) = \langle e_{ii}-e_{i+1,i+1},h\rangle$ for all $h\in\mathfrak{h}$.

The generalized Cremmer-Gervais $r$-matrices are the quasi-triangular $r$-matrices associated to the maximal BD-triples with $\#|\Pi-{\mathcal S}_0|=1$. In \cite{GG}, Gerstenhaber and Giaquinto show there are exactly $\phi(n)$ BD-triples of this type, where $\phi$ is the Euler-totient function. Additionally,  for $m$ coprime to $n$ in $\{1,...,n-1\}$ they show the corresponding BD-triple ${\mathcal T}_{m,n}$ is given by 
\begin{equation}\label{T_{m,n}}
{\mathcal T}_{m,n}:=(\Pi-\{\alpha_{n-m}\},\Pi-\{\alpha_m\},\zeta:\alpha_s\mapsto\alpha_{s+m(\text{mod }n)}).
\end{equation}
 
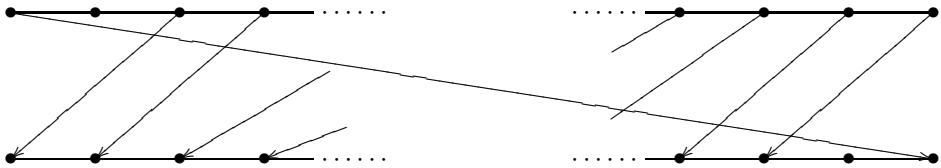
\begin{figure}[h]
\centering\caption{The Cremmer-Gervais BD-triple ${\mathcal T}_{n-2,n}$}\begin{picture}(300,75)\label{CG BD-triple T_n-2,n}\put(-48,60){\xymatrix @R=5pt @C=10pt{&*={\bullet}\ar[ddddrrrrrrrrrrrrrrrrrr]\ar@{-}[rrrrrr]&&*={\bullet}&&*={\bullet}\ar[ddddllll]&&*={\bullet}\ar[ddddllll]\ar@{-}[r]&\cdots\cdots&&&&\cdots\cdots&*={\bullet}\ar@{-}[l]\ar@{-}[rrrrrr]\ar@{-}[dl]&&*={\bullet}\ar@{-}[dddlll]&&*={\bullet}\ar[ddddllll]&&*={\bullet}\ar[ddddllll]&&&&&&&\\&*={}&*={}&&&&&&\ar[dddlll]\phantom{cdots\cdots}&&&&&&\\&&&&&&&&&*={\phantom{\cdots\cdots}}\\&&&&&&&&\ar[dl]&&&&&&&&&&&&&&&&&&&&\\&*={\bullet}\ar@{-}[rrrrrr]&&*={\bullet}&&*={\bullet}t&&*={\bullet}\ar@{-}[r]]&\cdots\cdots&&&&\cdots\cdots&*={\bullet}\ar@{-}[l]\ar@{-}[rrrrrr]&&*={\bullet}&&*={\bullet}&&*={\bullet}\\}}\end{picture}
\end{figure}

The generalized Cremmer-Gervais $r$-matrices are particularly interesting because the variety $\beta_{{\mathcal T}_{m,n}}$ is of minimal dimension. In fact, Gerstenhaber and Giaquinto \cite{GG} show that $\beta_{{\mathcal T}_{m,n}}$ is a singleton set containing the point
\begin{equation}
\beta_{m,n}:=\sum_{1\leq j<\ell\leq n}\left(-1+\frac{2}{n}[(j-\ell)m^{-1}(\text{mod }n)]\right)e_{jj}\wedge e_{\ell\ell}\in\mathfrak{h}\wedge\mathfrak{h}.
\end{equation}

In order to describe the $\alpha_{\mathcal T}$-part of the generalized Cremmer-Gervais $r$-matrices, which would consequently provide us with an explicit description of the entire $r$-matrix,  we need to first introduce some notation. Here and elsewhere, when an integer is reduced modulo $\ell\in\mathbb{N}$, we always mean in $\{0,...,\ell-1\}$.  Let $n$ and $m$ be relatively prime positive integers with $m<n$.  We define a sequence of integers $i_0,i_1,i_2,...$ recursively by setting $i_0 = n$, $i_1 = m$, and $i_{t}=-i_{t-2} (\text{mod } i_{t-1})$ for $t>1$.  Eventually the sequence will reach $1$. Let $L$ be the smallest number so that $i_L=1$. For all $0\leq t\leq L$ and $j,\ell\in\mathbb{Z}$, define 
\begin{equation}
{\mathcal J}_t(j,\ell):=1-i_t+\Big[(n-\ell)(\text{mod }i_0)(\text{mod }i_1)\cdots(\text{mod }i_t)\Big]+\Big[(j-1)(\text{mod }i_0)(\text{mod }i_1)\cdots(\text{mod }i_t)\Big],
\end{equation}
and let $\psi_j$ be the unique integer in $\{1,...,n\}$ satisfying $j = m\psi_j (\text{mod }n)$. Let $\text{sgn}:\mathbb{Z}\to\{-1,0,1\}$ denote the \emph{signum} function: $\text{sgn}(0)=0$ and for all nonzero $x$, $\text{sgn}(x)=\frac{|x|}{x}$. For $n\geq 3$ there is a non-trivial automorphism of the Dynkin diagram of $\mathfrak{sl}_n$. This corresponds to the involutive  automorphism of $\mathfrak{sl}_n$ given by $\varphi: e_{j\ell}\mapsto -e_{n+1-\ell,n+1-j}$. One can readily compute that $(\varphi\otimes\varphi)r_{m,n}=r_{n-m,n}$ and we have the following

\begin{theorem} \label{the action}
The  generalized Cremmer-Gervais $r$-matrix $r_{n-m,n}$ is given by the formula
\begin{align}\label{CG rmatrix}
r_{n-m,n}(e_j\otimes e_\ell)&=\sum_{t=0}^{L-1}\sum_{N=0}^{\lfloor\frac{{\mathcal J}_t(j,\ell)-1}{i_{t+1}}\rfloor}e_{j-{\mathcal J}_t(j,\ell)+Ni_{t+1}}\otimes e_{\ell +{\mathcal J}_t(j,\ell)-Ni_{t+1}}\\
\nonumber&\phantom{\mapsto}-\sum_{t=0}^{L-1}\sum_{N=0}^{\lfloor\frac{{\mathcal J}_t(\ell,j)-1}{i_{t+1}}\rfloor}e_{j +{\mathcal J}_t(\ell,j)-Ni_{t+1}}\otimes e_{\ell-{\mathcal J}_t(\ell,j)+Ni_{t+1}}\\
\nonumber&\phantom{\mapsto}+\Bigg[\frac{1}{2}\text{sgn}(\psi_j-\psi_\ell)-\frac{1}{n}(\psi_j-\psi_\ell)\Bigg]e_j\otimes e_\ell-\frac{1}{2}\text{sgn}(j-\ell)e_\ell\otimes e_j.
\end{align}
\end{theorem}
\begin{proof}
See Section \ref{33}.
\end{proof}

\subsection{Frobenius and Quasi-Frobenius Lie algebras}\label{Frobenius and Quasi-Frobenius Lie algebras}

There is not an analogue of the  Belavin-Drinfel'd classification theorem for triangular $r$-matrices. Instead, they are characterized by a homological condition \cite{Stolin}. Assume $r\in\mathfrak{g}\wedge\mathfrak{g}$ is a triangular $r$-matrix over the base field $k$. The \emph{carrier} of $r$ is the Lie subalgebra $\mathfrak{f}\subseteq\mathfrak{g}$ spanned by $\{(\xi\otimes 1)r:\xi\in\mathfrak{g}^*\}$. The map $\check{r}:\mathfrak{f}^*\to\mathfrak{f}$ defined by $\xi\mapsto (\xi\otimes 1)r$ is a linear isomorphism of $\mathfrak{f}$ with its dual $\mathfrak{f}^*$ and induces a nondegenerate skew-symmetric bilinear form
\begin{align}
F:\mathfrak{f}\times\mathfrak{f}&\to k \\
\nonumber(X,Y)&\mapsto \left<\check{r}^{-1}X,Y\right>.
\end{align} 
Since $r$ is a triangular $r$-matrix, it follows that $F$ is Lie algebra $2$-cocycle on $\mathfrak{f}$ with coefficients in the trivial representation:
\begin{align}\label{cocycle}
F([X,Y],Z)+F([Z,X],Y)+F([Y,Z],X)=0
\end{align}
for all $X,Y,Z\in\mathfrak{f}$. Here,  we call $(\mathfrak{f},F)$ a \emph{quasi-Frobenius Lie algebra}. The above process can be inverted to give a one-to-one correspondence between quasi-Frobenius Lie algebras $(\mathfrak{f},F)$ and \emph{nondegenerate} triangular $r$-matrices on ${\mathfrak f}$. Thus, a classification of triangular $r$-matrices for $\mathfrak{g}$ would entail classifying all quasi-Frobenius subalgebras $\mathfrak{f}\subseteq\mathfrak{g}$. One way to obtain a skew-symmetric bilinear form $F$ satisfying Eqn. \eqref{cocycle} is to choose an arbitrary functional $\eta\in{\mathfrak f}^*$ and set $F (X,Y)=\eta ([X,Y])$ $(X,Y\in{\mathfrak f})$.  In this setting $({\mathfrak f},\eta)$ is called a \emph{Frobenius Lie algebra} provided $F$ is nondegenerate. The functional $\eta\in\mathfrak{f}^*$ is called the \emph{Frobenius functional} of $\mathfrak{f}$.

\subsection{Maximal Parabolic Subalgebras of ${\mathfrak sl}_n$}
Certain maximal parabolic subalgebras of ${\mathfrak sl}_n$ provide interesting examples of Frobenius Lie algebras.  For $1\leq m<n$, let $\mathfrak{p}_{m,n}$ denote the maximal parabolic subalgebra of ${\mathfrak sl}_n$ obtained by deleting the $m$-th simple negative root $e_{m+1,m}$: i.e. 
\begin{equation*}
\mathfrak{p}_{m,n}=\text{span}_k\{e_{j\ell}:j\leq m \text{ or }m<\ell    \}\subseteq\mathfrak{sl}_n. 
\end{equation*}
The following theorem can be found in \cite[Theorem 5.3]{GG}. 

\begin{theorem}
The second cohomology group $H^2(\mathfrak{p}_{m,n},k)$ is trivial for all $1\leq m<n$. 
\end{theorem}

This implies the following.

\begin{corollary} 
A maximal parabolic subalgebra of ${\mathfrak sl}_n$ is quasi-Frobenius if and only if it is Frobenius.
\end{corollary}

The following theorem was proved by Elashvili \cite{Elashvili}.

\begin{theorem}\label{Elashvili} 
The maximal parabolic subalgebra ${\mathfrak p}_{m,n}$ of ${\mathfrak sl}_n$ is Frobenius if and only if $m$ and $n$ are relatively prime.
\end{theorem}

\subsection{Boundary $r$-matrices}

In \cite{GG}, it was shown that points lying in the Zariski boundary of the variety  ${\mathcal M}$ are triangular $r$-matrices.  Due to the close relationship that boundary $r$-matrices have with quasitriangular $r$-matrices, one would expect a classification result for boundary $r$-matrices that closely parallels the Belavin-Drinfeld classification. However very little is currently known about boundary $r$-matrices. The most general class of known examples include the \emph{generalized Jordanian $r$-matrices}, which are triangular $r$-matrices for the Lie algebra $\mathfrak{sl}_n$ having carrier $\mathfrak{p}_{1,n}\subseteq\mathfrak{sl}_n$.  In \cite{GG}, it was shown that the generalized Jordanian $r$-matrices lie on the boundary of the component of ${\mathcal M}$ corresponding to the Cremmer-Gervais $r$-matrices $r_{1,n}$.

This raises an interesting question and prompted Gerstenhaber and Giaquinto to make the following conjecture \cite[Conj. 5.7]{GG}.

\begin{conjecture}\label{Gerstenhaber-Giaquinto Conjecture} (The Gerstenhaber-Giaquinto Conjecture) 
Suppose $m$ and $n$ are relatively prime positive integers with $m<n$. Then the triangular $r$-matrix with carrier $\mathfrak{p}_{m,n}$ is a boundary $r$-matrix and lies in the closure of the $SL_n$-orbit of $r_{m,n}$.
\end{conjecture}

In \cite[Thm. 5.9]{GG}, they prove their conjecture in the case when $m=1$.  The following theorem provides a way to construct boundary $r$-matrices.

\begin{theorem} \cite[Prop. 5.1]{GG}\label{constructing boundary solutions}
Let $t\in k$ be nonzero. Suppose $r\in{\mathfrak g}\wedge{\mathfrak g}$ and $r_t = r+tr^\prime+\cdots+t^dr^{\prime\prime} \in{\mathfrak g}\wedge\mathfrak{g}$ are $r$-matrices  with $\langle\langle r,r\rangle\rangle=\langle\langle r_t,r_t\rangle\rangle$. Then $r^{\prime\prime}$ is a boundary $r$-matrix.
\end{theorem}

In particular, if $r$ is an $r$-matrix and $X\in\mathfrak{g}$ is nilpotent.  Then the highest degree term on $exp(tX).r$ is a boundary $r$-matrix. For instance, let 
\begin{eqnarray}\label{x}
X=\frac{1}{2}\left[(n-1)e_{12}+(n-2)e_{23}+\cdots+1\cdot e_{n-1,n}\right]\in{\mathfrak sl}_n.
\end{eqnarray}
Then $exp(tX).r_{1,n}=r_{1,n}+t[X,r_{1,n}]$. Thus, $[X,r_{1,n}]$ is a boundary $r$-matrix. We have the following

\begin{theorem} \cite[Thm. 5.9]{GG}\label{GG 5.9}
The boundary $r$-matrix $[X,r_{1,n}]$ lies in the closure of the $SL_n$-orbit of $r_{1,n}$. The carrier of $[X,r_{1,n}]$ is the maximal parabolic subalgebra $\mathfrak{p}_{1,n}\subseteq\mathfrak{sl}_n$.
\end{theorem}

The $r$-matrices $[X,r_{1,n}]$ are referred to as the \emph{generalized Jordanian $r$-matrices of Cremmer-Gervais type} (or the \emph{Jordanian $r$-matrices} for short)\cite{EH}. This is in reference to the fact that $[X,r_{1,2}]$ is the Jordanian $r$-matrix for ${\mathfrak sl}_2$.

\section{The Polynomial Representation of the Rational Cherednik Algebras, the algebra $\widetilde{H}$ of Differential-Dunkl Operators, and the Generalized Cremmer-Gervais r-matrices}\label{differential-Dunkl operators}

The aim of this section is to demonstrate how Eqn. \ref{CG rmatrix} above, which describes the action of the generalized Cremmer-Gervais r-matrices on $k^n\otimes k^n$, is related to the action of certain elements in the polynomial representation of the rational Cherednik algebras. To illustrate, when $m=1$, we have
\begin{equation}
\displaystyle{r_{n-1,n}(e_j\otimes e_\ell)=\text{sgn}(j-\ell)\Bigg(\frac{1}{2}(e_j\otimes e_\ell+e_\ell\otimes e_j)+\sum_{\tiny{\begin{array}{c}s\text{ strictly}\\ \text{between }j\\ \text{ and }\ell\end{array}}}e_s\otimes e_{j+\ell-s}\Bigg)-\frac{j-\ell}{n}e_j\otimes e_\ell}
\end{equation}
for all $j,\ell\in\{1,...,n\}$.
We let $TruncPol_n(k[x_1,x_2])$ denote the $k$-linear subspace of $k[x_1,x_2]$ spanned by the monomials $x_1^{j}x_2^{\ell}$ with $0\leq j,\ell<n$. Identifying the vector spaces $TruncPol_n(k[x_1,x_2])$ and $k^n\otimes k^n$ via $x_1^{j-1}x_2^{\ell-1}\leftrightarrow e_j\otimes e_\ell$ yields the formula 
\begin{equation}\label{CG action}
r_{n-1,n}= -\frac{1}{n}\left(x_1\partial_1-x_2\partial_2\right)+\frac{\Delta}{2},
\end{equation}
where $\Delta$ is the divided difference operator
\begin{equation}
\Delta :=\frac{x_1+x_2}{x_1-x_2} (1-\sigma),
\end{equation}
$\sigma$ is the operator that interchanges the variables $\sigma.f(x_1,x_2)=f(x_2,x_1)$ for all $f\in TruncPol_n(k[x_1,x_2])$, and $\partial_1$ and $\partial_2$ are the partial derivative operators with respect to $x_1$ and $x_2$.  In \cite{J}, it was shown that Eqn. \ref{CG action} above can be expressed in terms of the so-called Dunkl operators, which arise in the polynomial representations of rational Cherednik algebras.

\subsection{Rational Cherednik Algebras of Type G(m,1,2)}

First we review the construction of the rational Cherednik algebra $H_{\kappa, \bf c}(W)$ associated to a reflection group $W$ (and reflection representation $W\to GL(V)$, where $V$ is equipped with a nondegenerate inner product).  While most of the statements and results from this section are true over a base field $k$ containing a primitive $m$-th root of one and with $2n\neq 0$ (for a fixed choice of positive integers $m$ and $n$), we assume the base field $k$ is algebraically closed of characteristic zero. For our purposes, the only reflection group we are concerned with is $W=G(m,1,2)$, the group of $2\times 2$ monomial matrices having entries in $\{1,\omega,\omega^2,...,\omega^{m-1}\}$, where $\omega\in k$ is a primitive $m$-th root of unity. We may assume $V$ is the natural $2$-dimensional representation of $G(m,1,2)$. 
For a more detailed treatment on the rational Cherednik algebras, especially for the type we consider, we refer the reader to the work of Chmutova \cite{Chmut} and Griffeth \cite{Griff, Griff2}, where they discuss the representation theory of such algebras.  

Nondegeneracy of the inner product $\left(\,\,,\,\,\right)$ on $V$ induces an inner product on the dual vector space $V^*$, also denoted with $\left(\,\,,\,\,\right)$. A \emph{reflection} is an element $s\in W$ so that $\text{codim}[\text{Ker}(Id-s)]=1$. Let $ R$ denote the set of reflections in $W$. To each reflection $s$, let $h_s\subseteq V$ be the hyperplane fixed by $s$, and let $\alpha_s\in V^*$ be a nonzero functional that vanishes on $h_s$. Put $\alpha_s^{\vee}= \frac{2\left(\alpha_s,-\right)}{\left(\alpha_s,\alpha_s\right)}\in V^{**}\cong V$ and let $\{{\bf c}_s\}_{s\in R}$ be a collection of scalars in $k$ indexed by the reflections and satisfying ${\bf c}_{wsw^{-1}}={\bf c}_s$ for all $s\in  R$, $w\in W$. 
For us, the set of reflections is $\{\sigma\}\cup\{\xi_1^j,\xi_2^j,\xi_1^j\xi_2^{-j}\sigma:j=1,...,m-1\}$, where 
 \begin{align}
 \xi_1=\left(\begin{array}{cc}\omega&0\\ 0&1\end{array}\right),   & &\xi_2=\left(\begin{array}{cc} 1&0\\ 0&\omega\end{array}\right),  & &\sigma=\left(\begin{array}{cc}0&1\\ 1&0\end{array}\right).
 \end{align}
There are $m$ conjugacy classes of reflections. For $j\in\{1,...,m-1\}$, let ${\bf c}_j$ denote the scalar corresponding to the conjugacy class $\{\xi_1^j,\xi_2^j\}$, and let ${\bf c}_0$ denote the scalar associated to the conjugacy class $\{\sigma,\xi_1\xi_2^{-1}\sigma,...,\xi_1^{m-1}\xi_2^{1-m}\sigma\}$. 
Let $\kappa\in k$ be a fixed scalar.
Finally, let $ F$ be the $k$-algebra freely generated by the group algebra $kW$ and the symmetric algebras $k[V]$ and $k[V^*]$, where the natural inclusion maps $k[W]\to F$, $k[V]\to F$,  $k[V^*]\to F$ are algebra homomorphisms. The \emph{rational Cherednik algebra} $ H_{\kappa,\bf c}(W)$ is the quotient of $F$ by the relations
\begin{align}
&wxw^{-1}=x^w,& &wyw^{-1}=y^w,& &yx-xy=\kappa\left(x,y\right)-\sum_{s\in R}{\bf c}_s\left(\alpha_s,y\right)\left(x,\alpha_s^{\vee}\right)s,
\end{align}
for all $w\in W$, $x\in V^*$, and $y\in V$. The algebra $H_{\kappa,\bf c}(W)$ is $\mathbb{Z}$-graded with $\deg(x)=1$, $\deg(y)=-1$, $\deg(g)=0$ for every $x\in V^*$, $y\in V$, and $g\in kW$. For a $kW$-module $\Lambda$, define the \emph{Verma} module $M(\Lambda)$ by 
\begin{equation}
M(\Lambda):=\text{Ind}_{k[V^*]\rtimes kW}^{^{ H_{\bf c}(W)}}\Lambda,
\end{equation}
where $y\in k[V^*]$ acts via multiplication by $y(0)$ on $\Lambda$. When $\Lambda$ is the trivial representation {\bf 1}, we obtain the \emph{polynomial representation} $M({\bf 1})\cong k[V]$ and the elements of $V$ act via Dunkl operators,
\begin{equation}
y.f=\kappa\partial_yf-\sum_{s\in R}{\bf c}_s\left(\alpha_s,y\right)\frac{f-sf}{\alpha_s}
\end{equation}
for every $y\in V$ and $f\in k[V]$.

\begin{definition/proposition}
The rational Cherednik algebra $ H_{\kappa,\bf c}(G(m,1,2))$ is the $k$-algebra generated by $\sigma$, $\xi_1$, $\xi_2$, $x_1$, $x_2$, $y_1$, $y_2$ and has the following defining relations:
\begin{align}
&\label{R1}\sigma^2=\xi_1^m=\xi_2^m=1,&&\\
&\xi_i\xi_j=\xi_j\xi_i,\,\,\, x_ix_j=x_jx_i,\,\,\,y_iy_j=y_jy_i,&& (i,j\in\{1,2\}),& \\
&\xi_iy_i=\omega y_i\xi_i,\,\,\,\xi_i x_i=\omega^{-1}x_i\xi_i,&& (i=1,2),& \\
&\sigma x_i=x_{3-i}\sigma,\,\,\,\sigma y_i=y_{3-i}\sigma,\,\,\,\sigma\xi_i=\xi_{3-i}\sigma,&& (i=1,2),& \\
&\xi_ix_j=x_j\xi_i,\,\,\,\xi_iy_j=y_j\xi_i,&& (i\neq j),& \\
&\left[y_i,x_i\right]=\kappa-{\bf c}_0\sum_{r=0}^{m-1}\xi_1^r\xi_2^{-r}\sigma-\sum_{r=1}^{m-1}{\bf c}_r(1-\omega^{-r})\xi_i^r,&&    (i=1,2),& \\
&\label{R2}\left[y_1,x_2\right]={\bf c}_0\sum_{r=0}^{m-1}\omega^{-r}\xi_1^r\xi_2^{-r}\sigma,\hspace{.5cm}\left[y_2,x_1\right]={\bf c}_0\sum_{r=0}^{m-1}\omega^r\xi_1^r\xi_2^{-r}\sigma.&&
\end{align}
\end{definition/proposition}

In the polynomial representation of $H_{\kappa,\bf c}(G(m,1,2))$, the Dunkl operators act via the following formulas (c.f. \cite[Prop 2.2]{BB}):
\begin{align}
\label{Gm12y1}y_1.x_1^jx_2^\ell&\mapsto \kappa jx_1^{j-1}x_2^\ell-m{\bf c}_0\left(\sum_{N=0}^{\lfloor{\frac{j-\ell-1}{m}}\rfloor}x_1^{j-1-Nm}x_2^{\ell+Nm}-\sum_{N=1}^{\lfloor{\frac{\ell-j}{m}}\rfloor}x_1^{j-1+Nm}x_2^{\ell-Nm}\!\!\right)-\sum_{N=1}^{m-1}{\bf c}_N(1-\omega^{-Nj})x_1^{j-1}x_2^\ell,\\
\label{Gm12y2}y_2.x_1^jx_2^\ell&\mapsto \kappa\ell x_1^jx_2^{\ell-1}+m{\bf c}_0\left(\sum_{N=1}^{\lfloor{\frac{j-\ell}{m}}\rfloor}x_1^{j-Nm}x_2^{\ell-1+Nm}-\sum_{N=0}^{\lfloor{\frac{\ell-j-1}{m}}\rfloor}x_1^{j+Nm}x_2^{\ell-1-Nm}\right)-\sum_{N=1}^{m-1}{\bf c}_N(1-\omega^{-N\ell})x_1^jx_2^{\ell-1},
\end{align}
for all $j,\ell\geq 0$.

\subsection{The Algebra  $\widetilde{H}$ and the $r$-matrices {\bf $r_{n-1,n}$} and {\bf $r_{n-2,n}$}}
Given the similarity between the formulas for the actions of $r_{n-m,n}$ on $TruncPol_n(k[x_1,x_2])$ and the Dunkl operators for the algebra $ H_{\kappa,\bf c}(G(m,1,2))$ (compare Eqn. \ref{CG rmatrix} with Eqns. \ref{Gm12y1} and \ref{Gm12y2}), this suggests that one may reinterpret the Cremmer-Gervais $r$-matrices in terms of the algebras $H_{\kappa,\bf c}(G(m,1,2))$, or at least  modified versions of $H_{\kappa,\bf c}(G(m,1,2))$.   The following result verifies this statement for the case when $m=1$. 

\begin{theorem}\cite[Section 4]{J}
The action of the Cremmer-Gervais $r$-matrix $r_{n-1,n}$ on $k^n\otimes k^n$ coincides with the action of  $\frac{-1}{n}\left(x_1y_1-x_2y_2\right)\in H_{1,n/2}(G(1,1,2))$ on the truncated polynomial ring $TruncPol_n(k[x_1,x_2])$.
\end{theorem}
\begin{proof}
In the rational Cherednik algebra $H_{\kappa,\bf c}(G(1,1,2))$, the Dunkl operators are $y_1 = \kappa\partial_1 - {\bf c_0}\frac{1-\sigma}{x_1-x_2}$ and $y_2 = \kappa\partial_2 +{\bf c_0}\frac{1-\sigma}{x_1-x_2}$. The result follows from substituting these expressions into Eqn.\ref{CG action}.
\end{proof}

Throughout the remainder of this section, we focus on the case when $m=2$ and $n$ is odd. The main algebra we consider is the rational Cherednik algebra associated to the reflection group $G(2,1,2)$. For brevity, put 
\[H=H_{\kappa,\bf c}(G(2,1,2)).\]

A straightforward yet lengthy computation yields the following proposition, which we provide to give an example of an $r$-matrix arising from the polynomial representation of $H$.

\begin{proposition} Let ${\mathcal E}=x_1y_1-x_2y_2+ {\bf c}_0(\xi_1-\xi_2)\sigma\in H$.
As an operator on the polynomial representation $k[x_1,x_2]\cong k[x]\otimes k[x]$, we have $CYB_{4{\bf c}_0^2}({\mathcal E})=0$.
\end{proposition}

 Converting ${\mathcal E}$ into an element of  $\mathfrak{gl}_n\wedge\mathfrak{gl}_n$ yields
\begin{align*}
{\mathcal E} &=\sum_{1\leq j<\ell\leq n}\left(a_{j\ell}e_{jj}\wedge e_{\ell\ell}+b_{j\ell}e_{j\ell}\wedge e_{\ell j}\right)+\sum_{1\leq p<j<\ell\leq n}c_{j\ell}e_{j\ell}\wedge e_{\ell-p,j-p}
\end{align*}
where $a_{j\ell}, b_{j\ell}, c_{j\ell}$ are the constants
\begin{align*}
&a_{j\ell}=4{\bf c}_0-2\kappa(\ell-j)+2{\bf c}_1\left(\left(-1\right)^\ell-\left(-1\right)^j\right)\\
&b_{j\ell}=2{\bf c}_0\left(-1-(-1)^{j+\ell}+ (-1)^\ell - (-1)^j\right)\\
&c_{j\ell}=4{\bf c}_0\left(1+(-1)^{j+\ell}\right).
\end{align*}
Setting the deformation parameters to ${\bf c}_0=\frac{n}{4}$ and ${\bf c}_1=0$ is the only instance when ${\mathcal E}\in{\mathfrak sl}_n\wedge {\mathfrak sl}_n$.

In what follows we express $r_{n-2,n}$ in terms of operators in a slightly modified version of $H$.  Let $\widetilde{H}$ denote the $k$-algebra 
\begin{equation}
\widetilde{H} := H\left<x_1^{-1},x_2^{-1},\partial_1,\partial_2\right>.
\end{equation}
This is the algebra obtained from $H$ by adjoining the inverses of $x_1$ and $x_2$ as well as the partial derivative operators, $\partial_1$ and $\partial_2$. The commuting relations involving $x_1^{-1}$ and $x_2^{-1}$ can be obtained from the relations in $H$ involving $x_1$ and $x_2$. The partial derivatives $\partial_1$ and $\partial_2$ commute with each other and satisfy the following relations in $\widetilde{H}$:
\begin{equation*}
[\partial_j,x_\ell] = \delta_{j\ell}, \hspace{1cm} \partial_j\sigma = \sigma\partial_{3-j}, \hspace{1cm} \partial_j\xi_\ell = (-1)^{\delta_{j\ell}}\xi_\ell\partial_j,
\end{equation*}
for $j,\ell\in\{1,2\}$. The commuting relations involving the partial derivatives with the Dunkl operators $y_1$ and $y_2$ are less obvious, but we will not make use of them.  In Theorem \ref{m=2 case} below, we will demonstrate how the $r$-matrix $r_{n-2,n}$ can be interpreted in terms of the algebra $\widetilde{H}$. First we define the following elements of  $kW\subseteq\widetilde{H}$:
\begin{align*} 
&g_1=-\frac{\sigma}{4{\bf c_0}},
& &g_2=\frac{\kappa\sigma}{4{\bf c_0}}-\frac{1}{2n},
&  &g_3=\frac{1}{4}(1-\xi_1)(1-\xi_2),
& &g_4=-\frac{{\bf c_1}}{4{\bf c_0}}(\xi_1-\xi_2)\sigma.
\end{align*}

We have the following 

\begin{theorem}\label{m=2 case} Identifying $k^n\otimes k^n$ with $TruncPol_n(k[x_1,x_2])$ yields
\begin{equation}\label{555}
r_{n-2,n}=g_1(x_1y_1-x_2y_2)+g_2(x_1\partial_1-x_2\partial_2)+\left(\frac{x_2}{x_1}-\frac{x_1}{x_2}\right)g_3+g_4\in\widetilde{H}.
\end{equation}
\end{theorem}

\begin{proof}

When $m=2$,  Eqn. \ref{CG rmatrix} becomes
\begin{align*}
r_{n-2,n}(e_j\otimes e_\ell)&=
\sum_{N=0}^{\lfloor\frac{j-\ell-1}{2}\rfloor} e_{\ell+2N}\otimes e_{j-2N}-\sum_{N=0}^{\lfloor\frac{\ell-j-1}{2}\rfloor} e_{\ell-2N}\otimes e_{j+2N}\\
\nonumber&\phantom{===}+\frac{1}{4}(1+(-1)^j)(1+(-1)^\ell) \left(e_{j-1}\otimes e_{\ell+1}-e_{j +1}\otimes e_{\ell-1}\right)\\
\nonumber&\phantom{===}+\Bigg(\frac{1}{2}-\frac{1}{n}\left[\left((j-\ell)\frac{n+1}{2}\right)(\text{mod }n)\right]-\frac{1}{2}\delta_{j\ell}\Bigg)e_j\otimes e_\ell-\frac{1}{2}\text{sgn}(j-\ell)e_\ell\otimes e_j.
\end{align*}

The $\alpha$, $\beta$, and $\gamma$  parts of $r_{n-2,n}$ (refer to Eqns. \ref{a1234}-\ref{c1234}) act via 
\begin{align*}
\alpha &\longmapsto-\frac{M}{4}(1+\xi_1\xi_2)-\frac{1}{2}\sigma M +\frac{\Delta}{4}+\frac{1}{4}\xi_1\Delta\xi_2+\left(\frac{x_2}{x_1}-\frac{x_1}{x_2}\right)g_3,\\
\beta &\longmapsto\frac{M}{4}(1+\xi_1\xi_2)-\frac{1}{2n}(x_1\partial_1-x_2\partial_2),\\
\gamma &\longmapsto\frac{1}{2}\sigma M,
\end{align*}
where $M$ is the linear operator on $k[x_1^{\pm 1},x_2^{\pm 1}]$ defined by $M: x_1^jx_2^\ell\mapsto sgn(j-\ell)x_1^jx_2^\ell$. Summing $\alpha$, $\beta$, and $\gamma$ gives us 
\begin{equation*}
\label{CG2 formula}r_{n-2,n} =-\frac{1}{2n}(x_1\partial_1-x_2\partial_2)+\frac{1}{4}\left(\Delta+\xi_1\Delta\xi_2 +4\left(\frac{x_2}{x_1}-\frac{x_1}{x_2}\right)g_3\right).
\end{equation*}
In the algebra $H$, the Dunkl operators are $y_1=\kappa\partial_1-{\bf c_0}\frac{1-\sigma}{x_1-x_2}-{\bf c_0}\frac{1-\xi_1\xi_2\sigma}{x_1+x_2}-{\bf c_1}\frac{1-\xi_1}{x_1}$ and $y_2=\kappa\partial_2+{\bf c_0}\frac{1-\sigma}{x_1-x_2}-{\bf c_0}\frac{1-\xi_1\xi_2\sigma}{x_1+x_2}-{\bf c_1}\frac{1-\xi_2}{x_2}$. Substituting these into the above formula yields Eqn. \ref{555}.
\end{proof}

To conclude this section, we verify that $r_{n-2,n}$, expressed in terms of the algebra $\widetilde{H}$, is indeed an $r$-matrix. In fact, we have the following lemma, which is a more general result.

\begin{lemma} Let $\Delta = \frac{x_1+x_2}{x_1-x_2}(1-\sigma)$. Then
\begin{equation}
\label{3434}CYB_4\left(\Delta +\xi_1\Delta\xi_2+a_1\left(\frac{x_2}{x_1}-\frac{x_1}{x_2}\right)g_3+a_2(x_1\partial_1-x_2\partial_2)\right)=0
\end{equation}
for all $a_1,a_2\in k$.
\end{lemma}
\begin{proof}
Compute.
\end{proof}

\begin{corollary} We have $CYB_{\frac{1}{4}}(r_{n-2,n})=0$.
\end{corollary}
\begin{proof}  Setting $a_{1}=1$ and $a_2=\frac{2}{n}$ in Eqn. \ref{3434} coincides with $4r_{n-2,n}$. Thus, $CYB_4(4r_{n-2,n})=0$ (equivalently $CYB_{\frac{1}{4}}(r_{n-2,n})=0$).
\end{proof}

 \section{A Note on the Gerstenhaber-Giaquinto Conjecture: the case when m=n-2}\label{GG Conj Section}
 
In this section, we prove the Gerstenhaber-Giaquinto conjecture for the case when $m=n-2$ (and $n$ is necessarily odd).
First, we define the following:
\begin{align}
E_1 &:= \frac{1}{2}(x_1^{-1}\partial_1+x_2^{-1}\partial_2)-\frac{1}{4}(x_1^{-2}(1-\xi_1)+x_2^{-2}(1-\xi_2))\in\widetilde{H},\\
E_2 &:= \frac{1}{2}\left(x_1+x_2-x_1\xi_1-x_2\xi_2\right)\in\widetilde{H}.
 \end{align}

The truncated polynomial ring $TruncPol_n(k[x_1,x_2])$ is stable under the actions of $E_1$ and $E_2$. Under the vector space isomorphism $TruncPol_n(k[x_1,x_2])\cong k^n\otimes k^n$, $E_1$ and $E_2$ are identified with the matrices $\sum_{j=1}^{n-2}\left\lfloor \frac{j+1}{2}\right\rfloor e_{j,j+2}\in\mathfrak{sl}_n$ and $e_{n,n-1}+e_{n-2,n-3}+\cdots +e_{3,2}\in\mathfrak{sl}_n$ respectively. Let $\mathfrak{n}$ denote the Lie algebra generated by $E_1$ and $E_2$. One may readily check that the commutator $[E_1,E_2]$ is nonzero and commutes with $E_1$ and $E_2$. Hence,  $\mathfrak{n}$ is isomorphic to the Heisenberg Lie algebra. The defining relations of the universal enveloping algebra $U(\mathfrak{n})$ are $E_1^2E_2-2E_1E_2E_1+E_1E_2^2=0$ and $E_2^2E_1-2E_2E_1E_2+E_2E_1^2=0$. Since $r_{n-2,n}\in\mathfrak{sl}_n\wedge\mathfrak{sl}_n$, the adjoint action of $\mathfrak{sl}_n$ on $\mathfrak{sl}_n\wedge\mathfrak{sl}_n$ gives us a natural way to generate a $U(\mathfrak{n})$-module. A straightforward computation will verify the following lemma.

\begin{lemma}\label{Un module structure} 
$ $
\begin{enumerate}
\item The $U(\mathfrak{n})$-module generated by $r_{n-2,n}$ has a $k$-basis $\{r_{n-2,n},v_1,v_2,v_3,v_4\}$, where the vectors $v_1,...,v_4$ are given by the formulas
\begin{align}
\nonumber v_1&=\frac{(x_1x_2)^{-1}}{4}\left(\Delta-\xi_1\Delta\xi_2 +\xi_1-\xi_2 \right)-\frac{1}{4n} \left(  2(x_1^{-1}\partial_1-x_2^{-1}\partial_2)-x_1^{-2}(1-\xi_1)+x_2^{-2}(1-\xi_2)               \right)                        ,\\
\nonumber v_2&=\frac{1}{4}\left(x_2\xi_1-x_1\xi_2+(x_1-x_2)\xi_1\xi_2+\frac{1}{n}(x_1(1-\xi_1)-x_2(1-\xi_2))\right),\\
\nonumber v_3&=\frac{1}{2}(x_1x_2)^{-1}\left(x_1\xi_1-x_2\xi_2-(x_1-x_2)\xi_1\xi_2+\frac{1}{n}\left(x_2(1-\xi_1)-x_1(1-\xi_2)\right)\right),\\
\nonumber v_4&=\frac{1}{2}\left(\frac{x_2}{x_1}-\frac{x_1}{x_2}\right)(1-\xi_1-\xi_2+\xi_1\xi_2).
\end{align}
\item The $U(\mathfrak{n})$-module structure for $U(\mathfrak{n})r_{n-2,n}$ is given by  $E_1.r_{n-2,n}=v_1$, $E_1.v_2=\frac{1}{2}v_3$, $E_2.r_{n-2,n}=v_2$, $E_2.v_1=v_3$, $E_2.v_3=v_4$, and $E_1.v_1=E_1.v_3=E_1.v_4=E_2.v_2=E_2.v_4=0$.
\end{enumerate}
\end{lemma}
\begin{proof}
Parts (1) and (2) can be obtained by using the commutation relations among the generators of $\widetilde{H}$. Once (2) is established, it follows that $\{r_{n-2,n},v_1,v_2,v_3,v_4\}$ is a spanning set for $U(\mathfrak{n})r_{n-2,n}$. To obtain linearly independence, observe that $\widetilde{H}$ is $\mathbb{Z}$-graded with $\deg(x_i)=1$, $\deg(y_i)=\deg(\partial_i)=-1$, $\deg(\xi_i)=0$ for $i=1,2$.  Linear independence follows from the fact that $\deg(v_1)=-2$, $\deg(v_2)=1$, $\deg(v_3)=-1$ and $\deg(r_{n-2,n})=\deg(v_4)=0$ (but $r_{n-2,n}$ is not a $k$-multiple of $v_4$).
\end{proof}

Furthermore, we have the following.

\begin{proposition}\label{k-linear combo} Any $k$-linear combination of $v_1$, $v_2$, $v_3$, $v_4$ is a triangular $r$-matrix.
\end{proposition}
\begin{proof}
For a pair of linear operators $L^\prime,L^{\prime\prime}$ on the vector space $k[x_1^{\pm 1},x_2^{\pm 1}]\cong k[x^{\pm1}]\otimes k[x^{\pm1}]$, We introduce the notation $\langle\langle L^\prime,L^{\prime\prime}\rangle\rangle$ to mean
\[ \langle\langle L^\prime,L^{\prime\prime}\rangle\rangle := [L^\prime_{12},L^{\prime\prime}_{13}] +[L^\prime_{12},L^{\prime\prime}_{23}]+[L^\prime_{13},L^{\prime\prime}_{23}]\in End_k(k[x_1^{\pm1},x_2^{\pm1},x_3^{\pm3}])\cong End_k(k[x^{\pm1}]^{\otimes 3})\]
(in particular, $\langle\langle L,L\rangle\rangle = CYB_0(L)$ for all $L\in End_k(k[x_1^{\pm1},x_2^{\pm1}])$).
Let $r^\prime = e^{uE_2}e^{tE_1}r_{n-2,n}$ and $r^{\prime\prime} = e^{uE_1}e^{tE_2}r_{n-2,n}$. From Lemma \ref{Un module structure}, it follows that 
\[r^\prime=r_{n-2,n}+uv_1+tv_2+tuv_3+\frac{1}{2}t^2uv_4\]
 and 
 \[r^{\prime\prime} = r_{n-2,n}+uv_1+tv_2  + \frac{1}{2}tuv_3.\]
  Since $r^\prime$, $r^{\prime\prime}$ and $r_{n-2,n}$ are equivalent $r$-matrices, we have $\left<\left<r^\prime,r^\prime\right>\right>-\left<\left<r_{n-2,n},r_{n-2,n}\right>\right>=0$ and $\left<\left<r^{\prime\prime},r^{\prime\prime}\right>\right>-\left<\left<r_{n-2,n},r_{n-2,n}\right>\right>=0$. Expanding $\left<\left<r^\prime,r^\prime\right>\right>-\left<\left<r_{n-2,n},r_{n-2,n}\right>\right>$ and $\left<\left<r^{\prime\prime},r^{\prime\prime}\right>\right>-\left<\left<r_{n-2,n},r_{n-2,n}\right>\right>$ yields expressions of the form 
\begin{equation*}
0=uA_1+tA_2+tuA_3+t^2uA_4+u^2A_5+tu^2A_6+t^2u^2A_7+t^2A_8+t^3uA_9+t^4u^2A_{10}+t^3u^2A_{11}
\end{equation*}
and
\begin{equation*}
0=uB_1+tB_2+tuB_3+t^2uB_4+u^2B_5+tu^2B_6+t^2u^2B_7+t^2B_8
\end{equation*}
respectively, where $A_1,...,A_{11}$ and $B_1,...,B_8$ are certain cross-terms, e.g. $A_1 = \left<\left<r_{n-2,n},v_1\right>\right>+\left<\left<v_1,r_{n-2,n}\right>\right>$, $A_2=\left<\left<r_{n-2,n},v_2\right>\right>+\left<\left<v_2,r_{n-2,n}\right>\right>$, etc. It follows that $A_1=\cdots =A_{11}=0$ and $B_1=\cdots =B_8=0$. For $a,b,c,d\in k$, let $r_{abcd} = av_1+bv_2+cv_3+dv_4$. We have
\begin{align*}
\left<\left<r_{abcd},r_{abcd}\right>\right> 
& = a^2\left<\left<v_1,v_1\right>\right> + ab\left(\left<\left<v_1,v_2\right>\right>+\left<\left<v_2,v_1\right>\right>\right)+\cdots +d^2\left<\left<v_4,v_4\right>\right>\\
 & =  a^2A_5  + ab (2B_3-A_3) +acA_6+ad(2A_7-8B_7)+b^2A_8\\
 &\phantom{==}+2bcB_6+2bdA_0+4c^2B_7+2cdA_{11}+4d^2A_{10}\\
 & = 0.
\end{align*}
 \end{proof}

We conclude this section by showing that the Gerstenhaber-Giaquinto conjecture holds for the case when $m=n-2$. For parameters $u,t\in k$, define first the following:
\begin{equation}
b_{CG}(u,t) : = uv_1+tv_2+tuv_3+\frac{1}{2}t^2uv_4\in\widetilde{H}. 
\end{equation}
The vector space $TruncPol_n(k[x_1,x_2])$ is stable under the action of $b_{CG}(u,t)$, hence we can identify $b_{CG}(u,t)$ with an element of $\mathfrak{sl}_n\wedge\mathfrak{sl}_n$.  

The following theorem proves the Gerstenhaber-Giaquinto conjecture for the case when $m=n-2$. The proof of part 1 of Thm. \ref{GGConj} below was communicated to us by Anthony Giaquinto \cite{G_private_communication}.

\begin{theorem}\label{GGConj}
$ $
\begin{enumerate}
\item Let $u,t\in k$. Then 
\begin{equation}
e^{uE_2}e^{tE_1}.r_{n-2,n} = r_{n-2,n}+b_{CG}(u,t) .
\end{equation}
Hence  $b_{CG}(u,t) $ is a triangular $r$-matrix. For $u$ and $t$ not both zero, $b_{CG}(u,t)$ lies in the boundary component of the Cremmer-Gervais $r$-matrix $r_{n-2,n}$.
\item For $u,t\neq 0$, the carrier of $b_{CG}(u,t) $ (viewed as an element of $\mathfrak{sl}_n\wedge\mathfrak{sl}_n$) is the maximal parabolic subalgebra $\mathfrak{p}_{n-2,n}\subseteq\mathfrak{sl}_n$.
\item The Frobenius functional associated to $b_{CG}(u,t) $ is 
\[u^{-1}\left(e_{13}^*+e_{24}^*+\cdots +e_{n-2,n}^*\right) -t e_{n,n-1}^* +2 (e_{n-1,n-1}-e_{nn})^*\in \mathfrak{p}_{n-2,n}^*.\]
\end{enumerate}

\end{theorem}

\begin{proof}
Proposition \ref{k-linear combo} implies $b_{CG}(u,t) $ is a triangular $r$-matrix. For part (2), we compute. As operators on $k[x_1^{\pm 1},x_2^{\pm 1}]$, we have
\begin{align*}
v_1.x_1^jx_2^\ell &= \sum_{N=0}^{\left\lfloor\frac{j-\ell-2}{2}\right\rfloor} x_1^{\ell+2N}x_2^{j-2N-2}-\sum_{N=0}^{\left\lfloor\frac{\ell-j-2}{2}\right\rfloor}x_1^{\ell-2N-2}x_2^{j+2N}-\frac{1}{n}\left(\left\lfloor\frac{j}{2}\right\rfloor x_1^{j-2}x_2^\ell-\left\lfloor\frac{\ell}{2}\right\rfloor x_1^jx_2^{\ell-2}\right)\\
&\phantom{===}+([j\text{ even, }\ell\text{ odd, }j>\ell]-[j\text{ odd, }\ell\text{ even, }j<\ell])x_1^{j-1}x_2^{\ell-1},\\
v_2.x_1^jx_2^\ell &=\frac{1}{2}\left(     [\ell\text{ odd}]\left((-1)^j-\frac{1}{n}\right)x_1^jx_2^{\ell+1}-[j\text{ odd}]\left((-1)^\ell-\frac{1}{n}\right)x_1^{j+1}x_2^{\ell}\right),\\
v_3.x_1^jx_2^\ell &=  [\ell\text{ odd}]\left((-1)^j-\!\frac{1}{n}\right)x_1^jx_2^{\ell-1}-[j\text{ odd}]\left((-1)^\ell-\!\frac{1}{n}\right)x_1^{j-1}x_2^\ell,\\
v_4.x_1^jx_2^\ell &=2[j\text{ odd}, \ell\text{ odd}]\left(x_1^{j-1}x_2^{\ell+1}-x_1^{j+1}x_2^{\ell-1}\right),
\end{align*}
where $[{\mathcal P}]=1$ if the statement ${\mathcal P}$ is true and $[{\mathcal P}]=0$ if ${\mathcal P}$ is false.
Restricting these operators to the truncated polynomial ring $TruncPol_n(k[x_1,x_2])$ and then using the isomorphism $TruncPol_n(k[x_1,x_2])\cong k^n\otimes k^n$ translates the above formulas for $v_1,...,v_4$ into the following elements of $\mathfrak{sl}_n\wedge\mathfrak{sl}_n$:
\begin{align*}
v_1 &=  2\left(\sum_{1\leq \ell<j\leq n}\sum_{N=1}^{\left\lfloor\frac{j-\ell-1}{2}\right\rfloor}e_{\ell+2N-2,j}\wedge e_{j-2N,\ell}+\hspace{-.3cm}\sum_{\tiny\begin{array}{c}1\leq \ell<j\leq n\\ j\text{ odd, }\ell\text{ even}\end{array}}\hspace{-.3cm}e_{j-1,j}\wedge e_{\ell-1,\ell}+\sum_{j=1}^{n-2} e_{j,j+2}\wedge h_j\right),\\
v_2 &= 2E^-\wedge h_{n-1},\\
v_3 &= 4E^+\wedge h_{n-1},\\
v_4 &= 4E^+\wedge E^-,
\end{align*}
where $E^+ = e_{12}+e_{34}+\cdots +e_{n-2,n-1}$, $E^- = e_{n,n-1}+e_{n-2,n-3}+\cdots + e_{32}$, and 
\begin{equation*}\displaystyle{h_j := \sum_{N=0}^{\left\lfloor\frac{j-1}{2}\right\rfloor }e_{j-2N,j-2N}-\frac{1}{n}\left\lfloor\frac{j+1}{2}\right\rfloor \left(e_{1,1}+\cdots +e_{n,n}\right)}\end{equation*} for $j=1,...,n-1$. 
Following the notation from Section \ref{Frobenius and Quasi-Frobenius Lie algebras}, one can verify that the linear maps 
\[\check{r}:\mathfrak{p}_{n-2,n}^*\to\mathfrak{p}_{n-2,n}\]
 and 
\[\check{r}^{-1}:\mathfrak{p}_{n-2,n}\to\mathfrak{p}_{n-2,n}^*\]
are given by
\begin{align*}
\check{r}:e_{j\ell}^*\longmapsto &\begin{cases} u (e_{\ell-2,j}+e_{\ell-4,j-2}+\cdots ), & (\ell>j+2) \text{ or } (j\text{ even and } \ell=j+1),\\
                     -u(e_{\ell ,j+2}+e_{\ell +2,j+4}+\cdots ), &(\ell <j-1)\text{ or } ( j\text{ even and }\ell =j-1),\\
                     -u(e_{\ell ,j+2}+e_{\ell +2,j+4}+\cdots )+2tuh_{n-1}+t^2uE^-,&(j\text{ odd and } \ell =j+1),\\
                     -u(e_{\ell ,j+2}+e_{\ell +2,j+4}+\cdots ) +th_{n-1}-t^2uE^+,& (j\text{ odd and } \ell =j-1),\\
                     uh_j, & (\ell =j+2),
                      \end{cases}  \\
\check{r}:h_j^* \longmapsto  &\begin{cases}-ue_{j,j+2},& \hspace{1.95in}(j\neq n-1),\\
                             -tE^--2tuE^+,  &\hspace{1.95in}(j=n-1),\end{cases}\\
\check{r}^{-1}:e_{j\ell }    \longmapsto    &\begin{cases}  
 -2e_{n-1,n}^*-t^{-1}h_{n-1}^*-u^{-1}e_{n-3,n}^*,   &  \hspace{.86in} (j,\ell )=(n,n-1),\\
-u^{-1}h_j^*,  & \hspace{.86in} (\ell =j+2),\\
 u^{-1}e_{23}^*,     & \hspace{.86in} (j,\ell )=(1,2),\\
 2e_{n,n-1}^*-th_{n-1}^*-u^{-1}e_{n-2,n-1}^*,   &  \hspace{.86in} (j,\ell )=(n-1,n),\\
  -u^{-1}\left(e_{\ell -2,j}^*-e_{\ell ,j+2}^*\right),&\hspace{.86in} \text{otherwise},\end{cases}\\
\check{r}^{-1}:h_j             \longmapsto  &\begin{cases}u^{-1}e_{j,j+2}^*,  &\hspace{1.7in}(j\neq n-1),\\
                                                t^{-1}e_{n,n-1}^*+te_{n-1,n}^*, &\hspace{1.7in}(j=n-1).\end{cases}
\end{align*}
Put $\varphi = u^{-1}(e_{13}^*+e_{24}^*\cdots e_{n-2,n}^*)-te_{n,n-1}^*+2(e_{n-1,n-1}-e_{nn})^*\in\mathfrak{p}_{n-2,n}^*$. One can check that $\varphi [X,Y] = \left<X,\check{r}^{-1}(Y)\right>$ for all $X,Y\in\mathfrak{p}_{n-2,n}$.
\end{proof}

\section{Proof of Theorem \ref{the action}}\label{33}

As before, let $m$ and $n$ be relatively prime positive integers with $m<n$ and let ${\mathcal T}_{m,n}=({\mathcal S}_0,{\mathcal S}_1,\zeta)$ denote the unique BD-triple for $\mathfrak{sl}_n$ with ${\mathcal S}_1=\Pi-\{\alpha_m\}$. Let $\varphi:\mathfrak{sl}_n\to\mathfrak{sl}_n$ be the involutive Lie algebra automorphism defined by $e_{j\ell}\mapsto -e_{\ell^\prime,j^\prime}$ (where $t^\prime:=n+1-t$ for all $t\in\{1,...,n\}$). Furthermore, recall that the corresponding generalized Cremmer-Gervais $r$-matrix associated to ${\mathcal T}_{m,n}$ is $r_{m,n}=\alpha_{m,n}+\beta_{m,n}+\gamma_n$ where
\begin{align}
\alpha_{m,n}&:= 2\sum_{e_{j_1,j_2}\prec e_{j_3,j_4}} e_{j_1,j_2}\wedge e_{j_4,j_3},\\
\beta_{m,n}&:=\sum_{1\leq j<\ell\leq n}\left[-1+\frac{2}{n}[(j-\ell)m^{-1}(\text{mod }n)]\right]e_{jj}\wedge e_{\ell\ell},\\
\gamma_n &:= \sum_{1\leq j<\ell\leq n}e_{j\ell}\wedge e_{\ell j}.
\end{align}
Here the sum on $\alpha_{m,n}$ is over all appropriate quadruples of numbers $j_1,j_2,j_3,j_4$ with $e_{j_1,j_2}\prec e_{j_3,j_4}$.  The remainder of this section is devoted to giving a more precise description of such quadruples.  
Define first the sets
\begin{align}
{\mathcal S}_{m,n}(j_1,j_2) &:=\{s\in\{1,...,n\}\mid e_{j_1,s}\prec e_{j_1+j_2-s,j_2}\text{ for the BD-triple }{\mathcal T}_{m,n}\},\\
\overline{{\mathcal S}}_{m,n}(j_1,j_2) &:=\{s\in\{1,...,n\}\mid e_{j_1,s}\preceq e_{j_1+j_2-s,j_2}\text{ for the BD-triple }{\mathcal T}_{m,n}\}.
\end{align}

We readily compute $\alpha_{n-m,n}=(\varphi\otimes\varphi)\alpha_{m,n}$ and the action of $\alpha_{n-m,n}$ on $e_j\otimes e_\ell\in k^n\otimes k^n$ is given by
\begin{align}
\alpha_{n-m,n}(e_j\otimes e_\ell) &=\text{sgn}(\ell-j)e_\ell\otimes e_j+\sum_{s\in{\mathcal S}_{m,n}(j^\prime,\ell^\prime)}e_{s^\prime}\otimes e_{j+\ell-s^\prime}-\sum_{s\in{\mathcal S}_{m,n}(\ell^\prime,j^\prime)}e_{j+\ell-s^\prime}\otimes e_{s^\prime}\\
\nonumber&=\sum_{s\in\overline{{\mathcal S}}_{m,n}(j^\prime,\ell^\prime)}e_{s^\prime}\otimes e_{j+\ell-s^\prime}-\sum_{s\in\overline{{\mathcal S}}_{m,n}(\ell^\prime,j^\prime)}e_{j+\ell-s^\prime}\otimes e_{s^\prime}
\end{align}
Our goal now is to explicitly determine the elements of the sets $\overline{{\mathcal S}}_{m,n}(1,1),  \overline{{\mathcal S}}_{m,n}(1,2),...,   \overline{{\mathcal S}}_{m,n}(n,n)$. The running example we consider is the case when $n=31$ and $m=12$.  In Figure \ref{spinning wheel}, we've arranged the integers $1$ through $31$ in two concentric wheels which we view as having the ability to rotate independently around a central axis. In each wheel the numbers are arranged so that as one traverses counterclockwise, the numbers increase by $12$ modulo $31$. Spokes are placed after certain numbers, indicating when a reduction modulo $31$ occurs. The spokes partition the integers into \emph{strings}. Here, the strings are
\begin{align*}
&\{1,13,25\},\{6,18,30\},\{11,23\},\{4,16,28\},\{9,21\},\{2,14,26\},\\
&\{7,19,31\},\{12,24\},\{5,17,29\},\{10,22\},\{3,15,27\},\{8,20\}.
\end{align*}

\setlength{\unitlength}{2pt} 
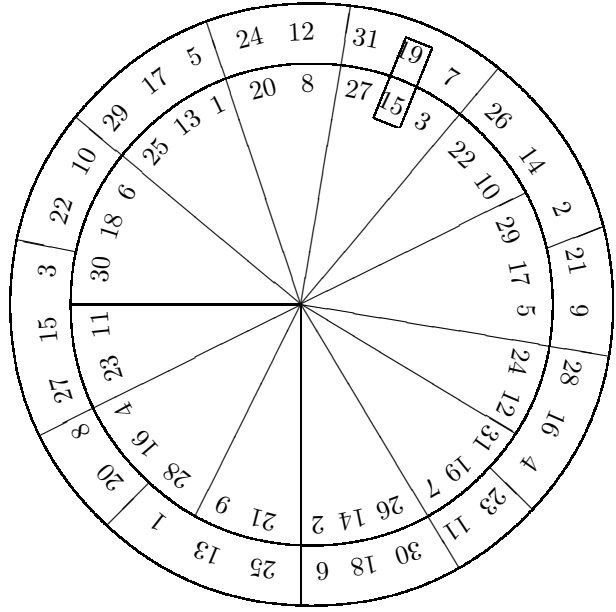
\begin{figure}[t]\centering
\caption{The Rotating Wheels with the pair $(15,19)$ aligned}
\begin{picture}(100,127)(0,-60)\label{spinning wheel}
\put(63.5, 34){\rotatebox{68}{\framebox{\phantom{xxxxx}}}}	
\put(89.18119765, 8.051940804){\rotatebox{ -78.38709677}{17}}
\put(86.75831246, 15.77423421){\rotatebox{-66.77419355}{29}}
\put(82.53053765, 23.55072861){\rotatebox{ -55.16129032}{10}}
\put(77.55867676, 28.99171149){\rotatebox{ -43.54838710}{22}}
\put(71.15856042, 33.94577030){\rotatebox{ -31.93548387}{3}}
\put(64.19221012, 37.81008528){\rotatebox{ -25.32258065}{15}}
\put(58.05711110, 39.53873297){\rotatebox{-11.709677419}{27}}
\put(49.97403326, 40.24866029){\rotatebox{ 2.903225806}{8}}
\put(39.97389873, 38.72308476){\rotatebox{14.51612903}{20}}
\put(32.38423396, 35.91218159){\rotatebox{26.12903226}{1}}
\put(25.51576070, 31.63102948){\rotatebox{ 37.74193548}{13}}
\put(19.64967510, 26.05489931){\rotatebox{ 49.35483871}{25}}
\put(15.02613536, 19.41207850){\rotatebox{60.96774194}{6}}
\put(11.83442974, 11.97452492){\rotatebox{72.58064516}{18}}
\put(10.20522706, 4.046732880){\rotatebox{ 84.19354839}{30}}
\put(10.20522706, -6.046732880){\rotatebox{ 95.80645161}{11}}
\put(11.83442974, -13.97452492){\rotatebox{ 107.4193548}{23}}
\put(14.52613536, -20.41207850){\rotatebox{119.0322581}{4}}
\put(17.64967510, -27.05489931){\rotatebox{130.6451613}{16}}
\put(23.51576070, -32.63102948){\rotatebox{142.2580645}{28}}
\put(33.38423396, -36.91218159){\rotatebox{ 153.8709677}{9}}
\put(39.97389873, -39.72308476){\rotatebox{ 165.4838710}{21}}
\put(51.97403326, -39.94866029){\rotatebox{ 181.0967742}{2}}
\put(57.05711110, -38.53873297){\rotatebox{191.7096774}{14}}
\put(63.89221012, -36.11008528){\rotatebox{200.3225806}{26}}
\put(72.85856042, -32.64577030){\rotatebox{209.9354839}{7}}
\put(76.75867676, -28.19171149){\rotatebox{ 219.5483871}{19}}
\put(82.33053765, -22.35072861){\rotatebox{235.1612903}{31}}
\put(86.75831246, -15.77423421){\rotatebox{246.7741935}{12}}
\put(89.18119765, -8.051940804){\rotatebox{258.3870968}{24}}
\put(91, 0){\rotatebox{ 270}{5}}
\put(99.97649706, 10.46492600){\rotatebox{-78.38709677}{21}}
\put(97.44789058, 18.61779276){\rotatebox{-66.77419355}{2}}
\put(91.03817206, 28.56341076){\rotatebox{-55.16129032}{14}}
\put(84.44834596, 36.23963936){\rotatebox{-43.54838710}{26}}
\put(76.44820052, 42.43221288){\rotatebox{-31.93548387}{7}}
\put(67.66526264, 47.28760660){\rotatebox{-23.32258065}{19}}
\put(59.57138888, 49.42341622){\rotatebox{-11.709677419}{31}}
\put(47.46754158, 49.93582536){\rotatebox{2.903225806}{12}}
\put(37.46737341, 48.40385595){\rotatebox{14.51612903}{24}}
\put(27.98029244, 44.89022699){\rotatebox{26.12903226}{5}}
\put(19.39470087, 39.53878684){\rotatebox{37.74193548}{17}}
\put(12.06209387, 32.56862414){\rotatebox{49.35483871}{29}}
\put(6.28266920, 24.26509813){\rotatebox{60.96774194}{10}}
\put(2.29303718, 14.96815616){\rotatebox{72.58064516}{22}}
\put(.25653383, 5.058416100){\rotatebox{ 84.19}{3}}
\put(.25653383, -7.058416100){\rotatebox{95.80}{15}}
\put(2.29303718, -17.96815616){\rotatebox{107.41}{27}}
\put(6.28266920, -24.26509813){\rotatebox{119.03}{8}}
\put(11.06209387, -33.56862414){\rotatebox{130.64}{20}}
\put(20.99470087, -40.53878684){\rotatebox{142.25}{1}}
\put(29.38029244, -46.39022699){\rotatebox{157.48}{13}}
\put(40.06737341, -48.70385595){\rotatebox{172.09}{25}}
\put(52.66754158, -48.73582536){\rotatebox{184.70}{6}}
\put(59.17138888, -47.32341622){\rotatebox{191.32}{18}}
\put(67.46526264, -44.68760660){\rotatebox{197.93}{30}}
\put(76.54820052, -39.53221288){\rotatebox{213.54}{11}}
\put(83.44834596, -34.33963936){\rotatebox{227.16}{23}}
\put(91.33817206, -28.16341076){\rotatebox{234.77}{4}}
\put(95.24789058, -19.51779276){\rotatebox{246.38}{16}}
\put(98.97649706, -10.06492600){\rotatebox{257}{28}}
\put(101, 0){\rotatebox{270}{9}}
\put(96.7,10.6){\line(5,2){10.6}}	
\put(7.6,10){\line(-5,1){11.2}}
\put(21.3,-33.7){\line(-1,-1){8}}
\put(85.8,-30.8){\line(1,-1){7.9}}
\put(50,0){\line(6,-1){58}}
\put(50,0){\line(2,1){42.5}}
\put(50,0){\line(5,6){37.3}}
\put(50,0){\line(1,6){9.4}}
\put(50,0){\line(-1,3){17.86}}
\put(50,0){\line(-6,5){42.55}}
\put(50,0){\line(-1,0){43.6}}
\put(50,0){\line(-2,-1){49.2}}
\put(50,0){\line(-1,-2){20}}
\put(50,0){\line(0,-1){57}}
\put(50,0){\line(3,-5){29.8}}
\put(50,0){\line(5,-3){40.7}}
\qbezier(97.6, 0)(97.6, 18.888)(84.2440, 32.2440)
\qbezier(84.2440, 32.2440)(70.888, 45.6)(52.0, 45.6)
\qbezier(52.0, 45.6)(33.112, 45.6)(19.7560, 32.2440)
\qbezier(19.7560, 32.2440)(6.4, 18.888)(6.4, 0)
\qbezier(6.4, 0)(6.4, -18.888)(19.7560, -32.2440)
\qbezier(19.7560, -32.2440)(33.112, -45.6)(52.0, -45.6)
\qbezier(52.0, -45.6)(70.888, -45.6)(84.2440, -32.2440)
\qbezier(84.2440, -32.2440)(97.6, -18.888)(97.6, 0)
\qbezier(109.0, 0)(109.0, 23.610)(92.3050, 40.3050)
\qbezier(92.3050, 40.3050)(75.610, 57.0)(52.0, 57.0)
\qbezier(52.0, 57.0)(28.390, 57.0)(11.6950, 40.3050)
\qbezier(11.6950, 40.3050)(-5.0, 23.610)(-5.0, 0)
\qbezier(-5.0, 0)(-5.0, -23.610)(11.6950, -40.3050)
\qbezier(11.6950, -40.3050)(28.390, -57.0)(52.0, -57.0)
\qbezier(52.0, -57.0)(75.610, -57.0)(92.3050, -40.3050)
\qbezier(92.3050, -40.3050)(109.0, -23.610)(109.0, 0)
\end{picture}
\end{figure}

Each string has a smallest integer, called the \emph{minimal element}. Here, the minimal elements of each string listed above are \[1,6,11,4,9,2,7,12,5,10,3,8\] respectively. For example, suppose $j=17$, $\ell=10$, and we wish to determine $\overline{{\mathcal S}}_{12,31}(17^\prime,10^\prime)$. Figure \ref{spinning wheel} indicates how $19\in\overline{{\mathcal S}}_{12,31}(17^\prime,10^\prime)$, for instance. First, we locate $j^\prime=15$ on the inner wheel and $\ell^\prime=22$ on the outer wheel.  As we traverse counterclockwise $8$ units from the aligned pair of numbers $(j^\prime,19)$, we see for instance that
\[e_{j^\prime,19}=e_{15,19}\prec e_{27,31}\prec \cdots \prec e_{6,10}\prec e_{18,22}=e_{18,\ell^\prime}.\]
Hence, $e_{j^\prime,19}\prec e_{18,\ell^\prime}$. In terms of the BD-triple ${\mathcal T}_{12,31}$ (recall Eqn. \ref{T_{m,n}}), this translates to
\[\widehat{\zeta}^{\, 8} (e_{e_{j^\prime}-e_{19}}) = e_{e_{18}-e_{\ell^\prime}}.\]

Notice however that $e_{j^\prime,19}\npreceq e_{11,j^\prime}$ because there does not exist $N\in\mathbb{Z}_{\geq 0}$ so that $\widehat{\zeta}^{\, N}(e_{e_{j^\prime}-e_{19}})=e_{e_{11}-e_{j^\prime}}$. This is indicated by the spokes on the two wheels not lining up correctly after traversing $9$ units counterclockwise from the pair $(j^\prime,19)$. However, there are other legitimate choices for $s$ with $e_{j^\prime ,s}\preceq e_{\ell^\prime+j^\prime-s,\ell^\prime}$. We can check all possibilities and see that $s=16,17,19,22$ are the only allowed values. However, we seek a more efficient algorithm and explicit formulae for finding all appropriate values of $s$ in $\overline{{\mathcal S}}_{m,n}(j^\prime,\ell^\prime)$. 

The first candidates we check are those in the same string as $\ell^\prime$. Thus, we test the integers of the form $\ell^\prime-Nm$ for $N=0,1,2,... $ The only condition that needs to be satisfied is $\ell^\prime -Nm> j^\prime$, or equivalently $N\leq\lfloor\frac{\ell^\prime - j^\prime-1}{m}\rfloor$. So we imagine first aligning $j^\prime$ on the inner wheel with $\ell^\prime$ on the outer wheel and then rotating the inner wheel clockwise so that $j^\prime$ aligns with $\ell^\prime-m$, then with $\ell^\prime - 2m$, etc. This will give us $1+\lfloor\frac{\ell^\prime - j^\prime-1}{m}\rfloor$ valid candidates (provided $\ell<j$). In our example, we check $s=22$ and $s=10$, but since $10<j^\prime$, this implies $10$ is not valid. Thus, so far this only gives us $s=22\in\overline{{\mathcal S}}_{12,31}(15,22)$.

To find the other $s\in\overline{{\mathcal S}}_{12,31}(15,22)$, we continue rotating the inner wheel clockwise so that the minimal element in the string directly counterclockwise to the string containing $j^\prime$ lines up with the minimal element of the string containing $\ell^\prime$. We denote these minimal elements by ${\mathcal A}(j^\prime)$ and ${\mathcal B}(\ell^\prime)$ respectively. They can be calculated according to the formulas
\begin{align}
{\mathcal A}(j^\prime)&:= m - \left[(n-j^\prime)(\text{mod }m)\right],\\
{\mathcal B} (\ell^\prime )&:=\ell^\prime-m\left\lfloor\frac{\ell^\prime-1}{m}\right\rfloor=1+\left[(\ell^\prime-1)(\text{mod }m)\right].
\end{align}
In our example, the string adjacent to the one containing $j^\prime$ is $\{8,20\}$, so we align $8$ with $10$ (see Figure \ref{spinning wheel 2}). 

This will reduce the problem to a smaller pair a relatively prime positive integers, $i_1$ and $i_2$, with $i_2<i_1$. Notice that we can naturally arrange the minimal elements on a pair of smaller rotating wheels, as in Figure \ref{spinning wheel 2}. As we circumnavigate the larger wheels counterclockwise, the minimal elements of each string are $1,6,11,4,9,2,7,12,5,10,3$ and $8$. The minimal elements are increasing by $-n(\text{mod }m)$. The spokes in the smaller wheel of minimal elements also tell us which minimal elements belong to large or small strings: the minimal elements of small strings have a spoke next to them in the counterclockwise direction. In particular, the spokes in the wheel of minimal elements tell us how the spokes align in the larger wheel.
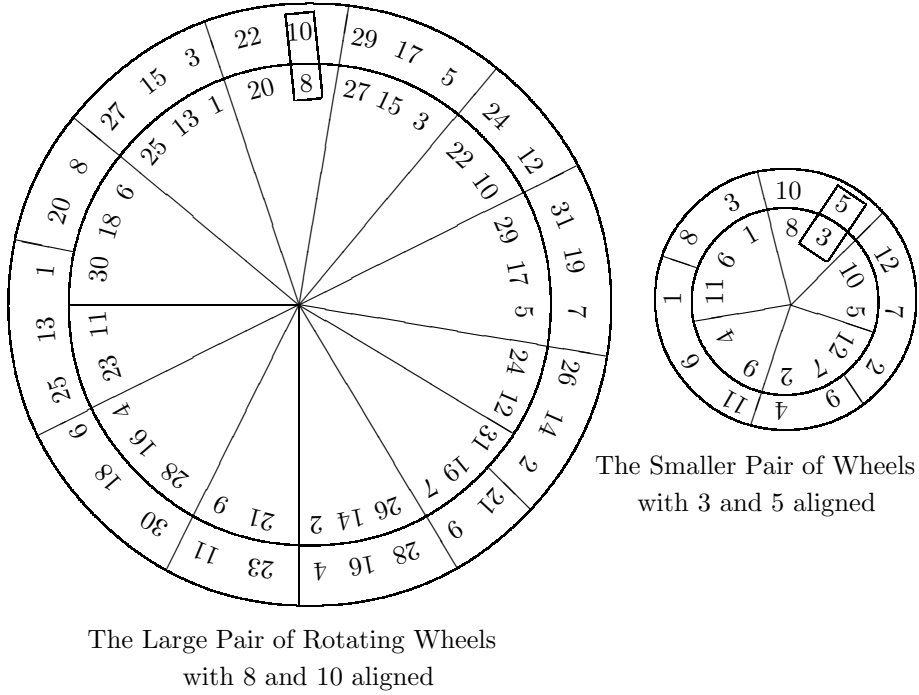
\begin{figure}[t]\centering
\caption{Reducing to a smaller pair of rotating wheels}
\begin{picture}(165,147)(0,-80)\label{spinning wheel 2}
\put(47.2, 39){\rotatebox{96}{\framebox{\phantom{xxxxx}}}}	
\put(89.18119765, 8.051940804){\rotatebox{ -78.38709677}{17}}
\put(86.75831246, 15.77423421){\rotatebox{-66.77419355}{29}}
\put(82.53053765, 23.55072861){\rotatebox{ -55.16129032}{10}}
\put(77.55867676, 28.99171149){\rotatebox{ -43.54838710}{22}}
\put(71.15856042, 33.94577030){\rotatebox{ -31.93548387}{3}}
\put(64.19221012, 37.81008528){\rotatebox{ -25.32258065}{15}}
\put(58.05711110, 39.53873297){\rotatebox{-11.709677419}{27}}
\put(49.97403326, 40.24866029){\rotatebox{ 2.903225806}{8}}
\put(39.97389873, 38.72308476){\rotatebox{14.51612903}{20}}
\put(32.38423396, 35.91218159){\rotatebox{26.12903226}{1}}
\put(25.51576070, 31.63102948){\rotatebox{ 37.74193548}{13}}
\put(19.64967510, 26.05489931){\rotatebox{ 49.35483871}{25}}
\put(15.02613536, 19.41207850){\rotatebox{60.96774194}{6}}
\put(11.83442974, 11.97452492){\rotatebox{72.58064516}{18}}
\put(10.20522706, 4.046732880){\rotatebox{ 84.19354839}{30}}
\put(10.20522706, -6.046732880){\rotatebox{ 95.80645161}{11}}
\put(11.83442974, -13.97452492){\rotatebox{ 107.4193548}{23}}
\put(14.52613536, -20.41207850){\rotatebox{119.0322581}{4}}
\put(17.64967510, -27.05489931){\rotatebox{130.6451613}{16}}
\put(23.51576070, -32.63102948){\rotatebox{142.2580645}{28}}
\put(33.38423396, -36.91218159){\rotatebox{ 153.8709677}{9}}
\put(39.97389873, -39.72308476){\rotatebox{ 165.4838710}{21}}
\put(51.97403326, -39.94866029){\rotatebox{ 181.0967742}{2}}
\put(57.05711110, -38.53873297){\rotatebox{191.7096774}{14}}
\put(63.89221012, -36.11008528){\rotatebox{200.3225806}{26}}
\put(72.85856042, -32.64577030){\rotatebox{209.9354839}{7}}
\put(76.75867676, -28.19171149){\rotatebox{ 219.5483871}{19}}
\put(82.33053765, -22.35072861){\rotatebox{235.1612903}{31}}
\put(86.75831246, -15.77423421){\rotatebox{246.7741935}{12}}
\put(89.18119765, -8.051940804){\rotatebox{258.3870968}{24}}
\put(91, 0){\rotatebox{ 270}{5}}
\put(99.97649706, 10.46492600){\rotatebox{-78.38709677}{19}}
\put(97.44789058, 18.61779276){\rotatebox{-66.77419355}{31}}
\put(91.03817206, 28.56341076){\rotatebox{-55.16129032}{12}}
\put(84.44834596, 36.23963936){\rotatebox{-43.54838710}{24}}
\put(76.44820052, 42.43221288){\rotatebox{-31.93548387}{5}}
\put(67.66526264, 47.28760660){\rotatebox{-23.32258065}{17}}
\put(59.57138888, 49.42341622){\rotatebox{-11.709677419}{29}}
\put(47.46754158, 49.93582536){\rotatebox{2.903225806}{10}}
\put(37.46737341, 48.40385595){\rotatebox{14.51612903}{22}}
\put(27.98029244, 44.89022699){\rotatebox{26.12903226}{3}}
\put(19.39470087, 39.53878684){\rotatebox{37.74193548}{15}}
\put(12.06209387, 32.56862414){\rotatebox{49.35483871}{27}}
\put(6.28266920, 24.26509813){\rotatebox{60.96774194}{8}}
\put(2.29303718, 14.96815616){\rotatebox{72.58064516}{20}}
\put(.25653383, 5.058416100){\rotatebox{ 84.19}{1}}
\put(.25653383, -7.058416100){\rotatebox{95.80}{13}}
\put(2.29303718, -17.96815616){\rotatebox{107.41}{25}}
\put(6.28266920, -24.26509813){\rotatebox{119.03}{6}}
\put(11.06209387, -33.56862414){\rotatebox{130.64}{18}}
\put(19.59470087, -41.53878684){\rotatebox{147.25}{30}}
\put(29.38029244, -46.39022699){\rotatebox{157.48}{11}}
\put(40.06737341, -48.70385595){\rotatebox{172.09}{23}}
\put(52.66754158, -48.73582536){\rotatebox{184.70}{4}}
\put(59.17138888, -47.32341622){\rotatebox{191.32}{16}}
\put(67.46526264, -44.68760660){\rotatebox{197.93}{28}}
\put(77.54820052, -40.53221288){\rotatebox{213.54}{9}}
\put(83.44834596, -34.33963936){\rotatebox{227.16}{21}}
\put(91.33817206, -28.16341076){\rotatebox{234.77}{2}}
\put(95.24789058, -19.51779276){\rotatebox{246.38}{14}}
\put(98.97649706, -10.06492600){\rotatebox{257}{26}}
\put(101, 0){\rotatebox{270}{7}}
\put(7.6,10){\line(-5,1){11.2}}
\put(85.8,-30.8){\line(1,-1){7.9}}
\put(144.3, 9){\rotatebox{56}{\framebox{\phantom{1111}}}}	
\put(50,0){\line(6,-1){58}}
\put(50,0){\line(2,1){52.5}}
\put(50,0){\line(5,6){37.3}}
\put(50,0){\line(1,6){9.4}}
\put(50,0){\line(-1,3){17.86}}
\put(50,0){\line(-6,5){42.55}}
\put(50,0){\line(-1,0){43.6}}
\put(50,0){\line(-2,-1){49.2}}
\put(50,0){\line(-1,-2){25}}
\put(50,0){\line(0,-1){57}}
\put(50,0){\line(3,-5){29.8}}
\put(50,0){\line(5,-3){40.7}}
\put(158.5, 10){\rotatebox{ -60}{12}}
\put(151, 18.32050808){\rotatebox{-30}{5}}
\put(140, 20){\rotatebox{ 0}{10}}
\put(130, 17.32050808){\rotatebox{30}{3}}
\put(121.6794919, 11){\rotatebox{ 60}{8}}
\put(119, 0){\rotatebox{ 90}{1}}
\put(122, -11){\rotatebox{ 120}{6}}
\put(129.4, -18.32050808){\rotatebox{150}{11}}
\put(140, -18.4){\rotatebox{ 180}{4}}
\put(149, -15.82050808){\rotatebox{ 210}{9}}
\put(157, -10){\rotatebox{ 240}{2}}
\put(161, 0){\rotatebox{ 270}{7}}	
\put(151.9583302, 7){\rotatebox{-60}{10}}
\put(147.4, 12){\rotatebox{-30}{3}}
\put(142, 13){\rotatebox{0.}{8}}
\put(133.5000000, 11.25833025){\rotatebox{ 30.}{1}}
\put(128.7416698, 6.5){\rotatebox{60.}{6}}
\put(127., 0){\rotatebox{90.}{11}}
\put(128.7416698, -6.5){\rotatebox{ 120.}{4}}
\put(133.5000000, -11.25833025){\rotatebox{ 150.}{9}}
\put(141, -12){\rotatebox{180.}{2}}
\put(146.5000000, -10){\rotatebox{ 210.}{7}}
\put(150.0583302, -4.6){\rotatebox{ 240.}{12}}
\put(153.7, 0){\rotatebox{ 270.}{5}}
\qbezier(166.7, 1.0)(166.7, 11.231)(159.4655, 18.4655)
\qbezier(159.4655, 18.4655)(152.231, 25.7)(142.0, 25.7)
\qbezier(142.0, 25.7)(131.769, 25.7)(124.5345, 18.4655)
\qbezier(124.5345, 18.4655)(117.3, 11.231)(117.3, 1.0)
\qbezier(117.3, 1.0)(117.3, -9.231)(124.5345, -16.4655)
\qbezier(124.5345, -16.4655)(131.769, -23.7)(142.0, -23.7)
\qbezier(142.0, -23.7)(152.231, -23.7)(159.4655, -16.4655)
\qbezier(159.4655, -16.4655)(166.7, -9.231)(166.7, 1.0)
\qbezier(159.57, .40)(159.57, 7.9191)(154.39455, 13.09455)
\qbezier(154.39455, 13.09455)(149.2191, 18.27)(141.90, 18.27)
\qbezier(141.90, 18.27)(134.5809, 18.27)(129.40545, 13.09455)
\qbezier(129.40545, 13.09455)(124.23, 7.9191)(124.23, .40)
\qbezier(124.23, .40)(124.23, -6.7191)(129.40545, -11.89455)
\qbezier(129.40545, -11.89455)(134.5809, -17.07)(141.90, -17.07)
\qbezier(141.90, -17.07)(149.2191, -17.07)(154.39455, -11.89455)
\qbezier(154.39455, -11.89455)(159.57, -6.7191)(159.57, .40)
\put(143,0){\line(3,-1){15.5}}
\put(143,0){\line(1,1){17.5}}	
\put(143,0){\line(-1,4){6.25}}	
\put(143,0){\line(-6,-1){18}}
\put(143,0){\line(-1,-3){7.5}}
\put(125.4,7){\line(-3,1){6.65}}
\put(153,-13.2){\line(2,-3){3.82}}
\put(106,-32){The Smaller Pair of Wheels}
\put(114,-39){with $3$ and $5$ aligned}
\put(10,-65){The Large Pair of Rotating Wheels}
\put(28,-72){with $8$ and $10$ aligned}
\qbezier(97.6, 0)(97.6, 18.888)(84.2440, 32.2440)
\qbezier(84.2440, 32.2440)(70.888, 45.6)(52.0, 45.6)
\qbezier(52.0, 45.6)(33.112, 45.6)(19.7560, 32.2440)
\qbezier(19.7560, 32.2440)(6.4, 18.888)(6.4, 0)
\qbezier(6.4, 0)(6.4, -18.888)(19.7560, -32.2440)
\qbezier(19.7560, -32.2440)(33.112, -45.6)(52.0, -45.6)
\qbezier(52.0, -45.6)(70.888, -45.6)(84.2440, -32.2440)
\qbezier(84.2440, -32.2440)(97.6, -18.888)(97.6, 0)
\qbezier(109.0, 0)(109.0, 23.610)(92.3050, 40.3050)
\qbezier(92.3050, 40.3050)(75.610, 57.0)(52.0, 57.0)
\qbezier(52.0, 57.0)(28.390, 57.0)(11.6950, 40.3050)
\qbezier(11.6950, 40.3050)(-5.0, 23.610)(-5.0, 0)
\qbezier(-5.0, 0)(-5.0, -23.610)(11.6950, -40.3050)
\qbezier(11.6950, -40.3050)(28.390, -57.0)(52.0, -57.0)
\qbezier(52.0, -57.0)(75.610, -57.0)(92.3050, -40.3050)
\qbezier(92.3050, -40.3050)(109.0, -23.610)(109.0, 0)
\end{picture}
\end{figure}

Next, we define integers
\begin{align}
{\mathcal C}_1 (\ell^\prime)&:=m+1-{\mathcal B}(\ell^\prime)=m-\left[(\ell^\prime -1)(\text{mod }m)\right]\\
{\mathcal D}_1(j^\prime)&:=m+1-{\mathcal A}(j^\prime) =1+\left[(n-j^\prime)(\text{mod }m)\right]
\end{align}
and determine all values $\widetilde{s}\in\overline{{\mathcal S}}_{-n(\text{mod }m),m}({\mathcal C}_1 (\ell^\prime),{\mathcal D}_1(j^\prime))$, i.e. reduce to the smaller pair of wheels.  Continuing with our example, we seek all $\widetilde{s}\in\overline{{\mathcal S}}_{5,12}(3,5)$. We get $\overline{{\mathcal S}}_{5,12}(3,5)=\{4,5,7\}$. However, since  $e_{3,4}\preceq e_{4,5}$, $e_{3,5}\preceq e_{3,5}$, and $e_{3,7}\preceq e_{1,5}$ hold in the ${\mathcal T}_{5,12}$-case, this will imply  $e_{15,16}\preceq e_{21,22}$, $e_{15,17}\preceq e_{20,22}$, and $e_{15,19}\preceq e_{18,22}$ for the ${\mathcal T}_{12,31}$-case.

In general, the approach we take is to check all candidates in the same string containing $\ell^\prime$, then reduce the problem to a smaller case. Thus,
\begin{equation}
\sum_{s\in\overline{{\mathcal S}}_{m,n}(j^\prime,\ell^\prime)}e_s =\underbrace{\sum_{N=0}^{\lfloor\frac{\ell^\prime -j^\prime -1}{m}\rfloor} e_{\ell^\prime - Nm}}_{\text{in the same string as $\ell^\prime$}}+\underbrace{\sum_{s\in\overline{{\mathcal S}}_{-n(\text{mod }m),m}({\mathcal C}_1 (\ell^\prime),{\mathcal D}_1(j^\prime))}e_{j^\prime+(s-{\mathcal C}_1 (\ell^\prime))}}_{\text{everything else}}.
\end{equation}
We define a sequence of integers $i_0,i_1,i_2,...$ recursively by $i_0:=n$, $i_1:=m$, and $i_t:=-i_{t-2} (\text{mod }i_{t-1})$ for $t>1$. Eventually the sequence will reach $1$ (e.g. $31,12,5,3,1$). Let $L$ denote the smallest number so that $i_L=1$. With this setup, we recursively get
\begin{align}\label{56}
\displaystyle{\sum_{s\in\overline{{\mathcal S}}_{m,n}(j^\prime,\ell^\prime)}e_s} &=\sum_{N=0}^{\lfloor\frac{{\mathcal D}_0(\ell^\prime)-{\mathcal C}_0(j^\prime)-1}{i_1}\rfloor} e_{j^\prime+{\mathcal D}_0(\ell^\prime)-{\mathcal C}_0(j^\prime)- Ni_1}+ \sum_{N=0}^{\lfloor\frac{{\mathcal D}_1(j^\prime)-{\mathcal C}_1(\ell^\prime)-1}{i_2}\rfloor} e_{j^\prime+{\mathcal D}_1(j^\prime)-{\mathcal C}_1(\ell^\prime)- Ni_2}\\
\nonumber&\phantom{=}+ \sum_{N=0}^{\lfloor\frac{{\mathcal D}_2(\ell^\prime)-{\mathcal C}_2(j^\prime)-1}{i_3}\rfloor} e_{j^\prime+{\mathcal D}_2(\ell^\prime)-{\mathcal C}_2(j^\prime)- Ni_3}+ \sum_{N=0}^{\lfloor\frac{{\mathcal D}_3(j^\prime)-{\mathcal C}_3(\ell^\prime)-1}{i_4}\rfloor} e_{j^\prime+{\mathcal D}_3(j^\prime)-{\mathcal C}_3(\ell^\prime)- Ni_4}+\cdots
\end{align}
where
\begin{align}
{\mathcal C}_t(\ell)&:=\begin{cases}i_t-\left((i_0-\ell)(\text{mod }i_0)(\text{mod }i_1)\cdots (\text{mod }i_t)\right),&(t \text{ is even}),\\
i_t-\left((\ell-1)(\text{mod }i_0)(\text{mod }i_1)\cdots (\text{mod }i_t)\right),&(t\text{ is odd}),\end{cases}\\
{\mathcal D}_t(\ell)&:=\begin{cases}1+\left((\ell-1)(\text{mod }i_0)(\text{mod }i_1)\cdots (\text{mod }i_t)\right),&(t \text{ is even}),\\
1+\left((i_0-\ell)(\text{mod }i_0)(\text{mod }i_1)\cdots (\text{mod }i_t)\right),&(t\text{ is odd}),\end{cases}
\end{align}
for all $0\leq t\leq L-1$ and $1\leq\ell\leq n$. After simplifying Eqn. \ref{56}, we get
\begin{equation}
\sum_{s\in\overline{{\mathcal S}}_{m,n}(j^\prime,\ell^\prime)}e_s=\sum_{t=0}^{L-1}\sum_{N=0}^{\lfloor\frac{{\mathcal J}_t(j,\ell)-1}{i_{t+1}}\rfloor}e_{j^\prime +{\mathcal J}_t(j,\ell)-Ni_{t+1}}
\end{equation}
where
\begin{align}
{\mathcal J}_t(j,\ell)&:=\begin{cases}{\mathcal D}_t(\ell^\prime)-{\mathcal C}_t(j^\prime),&(t\text{ is even}),\\ {\mathcal D}_t(j^\prime)-{\mathcal C}_t(\ell^\prime),&(t\text{ is odd}),\end{cases}\\
\nonumber&=1-i_t+\Big[(n-\ell)(\text{mod }i_0)(\text{mod }i_1)\cdots(\text{mod }i_t)\Big]+\Big[(j-1)(\text{mod }i_0)(\text{mod }i_1)\cdots(\text{mod }i_t)\Big].
\end{align}
Therefore,
\begin{align}
\label{a1234}\alpha_{n-m,n}(e_j\otimes e_\ell)&= \text{sgn}(\ell-j)e_\ell\otimes e_j +\sum_{t=0}^{L-1}\sum_{N=0}^{\left\lfloor\frac{{\mathcal J}_t(j,\ell)-1}{i_{t+1}}\right\rfloor}e_{j-{\mathcal J}_t(j,\ell)+Ni_{t+1}}\otimes e_{\ell+{\mathcal J}_t(j,\ell)-Ni_{t+1}}\\
\nonumber&\phantom{=}-\sum_{t=0}^{L-1}\sum_{N=0}^{\lfloor\frac{{\mathcal J}_t(\ell,j)-1}{i_{t+1}}\rfloor}e_{j +{\mathcal J}_t(\ell,j)-Ni_{t+1}}\otimes e_{\ell-{\mathcal J}_t(\ell,j)+Ni_{t+1}}.
\end{align}

Recall that $\displaystyle{\beta_{m,n} = \sum_{1\leq j<\ell\leq n}\left[-1+\frac{2}{n}[(j-\ell)m^{-1}(\text{mod }n)]\right]e_{jj}\wedge e_{\ell\ell}}$ and $\displaystyle{\gamma_n = \sum_{1\leq j<\ell\leq n}e_{j\ell}\wedge e_{\ell j}}$. Applying the automorphism $\varphi\otimes\varphi$ yields $(\varphi\otimes\varphi)\beta_{m,n}=-\beta_{m,n}=\beta_{n-m,n}$ and $(\varphi\otimes\varphi)\gamma_n = \gamma_n$. Thus,
\begin{align}
\label{b1234}\beta_{n-m,n}(e_j\otimes e_\ell)&=\Bigg(\frac{1}{2}-\frac{1}{n}\Big[((j-\ell)m^{-1})(\text{mod }n)\Big]-\frac{1}{2}\delta_{j\ell}\Bigg)e_j\otimes e_\ell,\\
\label{c1234}\gamma_n(e_j\otimes e_\ell) &=\frac{1}{2}\text{sgn}(j-\ell)e_\ell\otimes e_j.
\end{align}

For $j\in\{1,...,n\}$,  let $\psi_j$ denote the unique integer in $\{1,..,n\}$ satisfying $j =m\psi_j (\text{mod }n)$. With this notation, the action of $r_{n-m,n}$ is 
\begin{align}
r_{n-m,n}(e_j\otimes e_\ell)&=\sum_{t=0}^{L-1}\sum_{N=0}^{\lfloor\frac{{\mathcal J}_t(j,\ell)-1}{i_{t+1}}\rfloor}e_{j-{\mathcal J}_t(j,\ell)+Ni_{t+1}}\otimes e_{\ell +{\mathcal J}_t(j,\ell)-Ni_{t+1}}\\
\nonumber&\phantom{\mapsto}-\sum_{t=0}^{L-1}\sum_{N=0}^{\lfloor\frac{{\mathcal J}_t(\ell,j)-1}{i_{t+1}}\rfloor}e_{j +{\mathcal J}_t(\ell,j)-Ni_{t+1}}\otimes e_{\ell-{\mathcal J}_t(\ell,j)+Ni_{t+1}}\\
\nonumber&\phantom{\mapsto}+\Bigg[\frac{1}{2}\text{sgn}(\psi_j-\psi_\ell)-\frac{1}{n}(\psi_j-\psi_\ell)\Bigg]e_j\otimes e_\ell-\frac{1}{2}\text{sgn}(j-\ell)e_\ell\otimes e_j.
\end{align}

\end{document}